\documentclass[smallextended,hidelinks]{svjour3}

\smartqed  

\usepackage{geometry}
\usepackage{amssymb}
\usepackage{mathtools}
\usepackage{mathptmx}
\usepackage{float}
\usepackage{graphicx}
\usepackage{color}
\usepackage{caption}
\usepackage{multirow}
\usepackage{array}
\usepackage{fancyvrb}
\usepackage{hyperref}
\usepackage[labelformat=simple]{subcaption}

\newcommand{\real}{\mathbb{R}}
\newcommand{\tr}{\top}

\DeclareMathOperator*{\argmax}{arg\,max}
\DeclareMathOperator*{\argmin}{arg\,min}

\newtheorem{assumption}{Assumption}

\def\MCP{\text{MCP}}
\def\PATH{\textsc{Path}}
\def\SOL{\text{SOL}}
\def\QVI{\text{QVI}}
\def\VI{\text{VI}}

\newcolumntype{L}[1]{>{\raggedright\let\newline\\\arraybackslash\hspace{0pt}}m{#1}}
\newcolumntype{C}[1]{>{\centering\let\newline\\\arraybackslash\hspace{0pt}}m{#1}}
\newcolumntype{R}[1]{>{\raggedleft\let\newline\\\arraybackslash\hspace{0pt}}m{#1}}

\title{Solving equilibrium problems using extended mathematical programming}
\author{Youngdae~Kim \and Michael~C.~Ferris}

\institute{Y. Kim $\cdot$ M.C. Ferris \at Department of Computer Sciences and
  Wisconsin Institute for Discovery,
  University of Wisconsin-Madison, 1210 West Dayton St., Madison, WI, 53706\\
  \and
  Y. Kim \\
  \email{youngdae@cs.wisc.edu}
  \\
  M.C. Ferris \\
  \email{ferris@cs.wisc.edu}
}

\date{Received: date / Accepted: date}

\usepackage{listings}

\lstdefinelanguage{GAMS}{
morekeywords={
ABORT , ACRONYM , ACRONYMS , ALIAS , ALL , AND , ASSIGN , BINARY , CARD , DISPLAY , EPS , EQ , EQUATION , EQUATIONS , GE , GT , INF , INTEGER , LE , LOOP , LT , MAXIMIZING , MINIMIZING , MODEL , MODELS , NA , NE , NEGATIVE , NOT , OPTION , OPTIONS , OR , ORD , PARAMETER , PARAMETERS , POSITIVE , PROD , SCALAR , SCALARS , SET , SETS , SMAX , SMIN , SOS1 , SOS2 , SUM , SYSTEM , TABLE , USING , VARIABLE , VARIABLES , XOR , YES , REPEAT , UNTIL , WHILE , IF , THEN , ELSE , SEMICONT , SEMIINT , FILE , FILES , PUT , PUTPAGE , PUTTL , PUTCLOSE , FREE , NO , SOLVE , FOR , ELSEIF , ABS , ARCTAN , CEIL , COS , ERROR , EXP , FLOOR , LOG , LOG10 , MAP , MAPVAL , MAX , MIN , MOD , NORMAL , POWER , ROUND , SIGN , SIN , SQR , SQRT , TRUNC , UNIFORM , LO , UP , FX , SCALE , PRIOR , PC , PS , PW , TM , BM , CASE , DATE , IFILE , OFILE , PAGE , RDATE , RFILE , RTIME , SFILE , TIME , TITLE , TS , TL , TE , TF , LJ , NJ , SJ , TJ , LW , NW , SW , TW , ND , NR , NZ , CC , HDCC , TLCC , LL , HDLL , TLLL , LP , WS , /,PROD: },
sensitive = false,
morecomment=[f][\color{red}]*,%
morecomment=[s]{$ontext}{$offtext},
morestring=[b]'',
morestring=[b]',
}
\begin{document}
\maketitle

\begin{abstract}
  We introduce an extended mathematical programming framework for specifying
  equilibrium problems and their variational representations, such as
  generalized Nash equilibrium, multiple optimization problems with
  equilibrium constraints, and (quasi-) variational inequalities, and
  computing solutions of them from modeling languages.
  We define a new set of constructs with which users annotate variables and
  equations of the model to describe equilibrium and variational problems.
  Our constructs enable a natural translation of the model from one
  formulation to another more computationally tractable form without requiring
  the modeler to supply derivatives.
  In the context of many independent agents in the equilibrium, we facilitate
  expression of sophisticated structures such as shared constraints and
  additional constraints on their solutions.
  We define a new concept, shared variables, and demonstrate its uses for
  sparse reformulation, equilibrium problems with equilibrium constraints,
  mixed pricing behavior of agents, and so on.
  We give some equilibrium and variational examples from the literature and
  describe how to formulate them using our framework.
  Experimental results comparing performance of various complementarity
  formulations for shared variables are given.
  Our framework has been implemented and is available within GAMS/EMP.
\end{abstract}

\section{Introduction}
\label{sec:intro}

In this paper, we present an extended mathematical programming (EMP)
framework for specifying equilibrium problems and their variational
representations and computing solutions of them in modeling languages such as
AMPL, GAMS, or Julia~\cite{bezanson17,brook88,fourer02}.
Equilibrium problems of interest are (generalized) Nash equilibrium problems
(GNEP) and multiple optimization problems with equilibrium constraints
(MOPEC), and we consider quasi-variational inequalities (QVI) in their
variational forms.
All of these problems have been used extensively in the literature, see for
example~\cite{britz13,facchinei10,harker91,philpott16}.

The GNEP is a Nash game between agents with non-disjoint strategy sets.
For a given number of agents $N$, $x^*=(x_1^*,\dots,x_N^*)$ is a solution to
the GNEP if it satisfies
\begin{equation}
  \begin{aligned}
    x_i^* \in \argmin_{x_i \in K_i(x^*_{-i}) \subset \real^{n_i}} &&
    f_i(x_i,x^*_{-i}), && \text{for} \quad i=1,\dots,N,
  \end{aligned}
  \tag{GNEP}
  \label{eq:gnep}
\end{equation}
where $f_i(x_i,x_{-i})$ is the objective function of agent $i$, and
$K_i(x_{-i})$ is its feasible region.
Note that the objective function and the feasible region of each agent are
affected by the decisions of other agents, denoted by
$x_{-i}=(x_1,\dots,x_{i-1},x_{i+1},\dots,x_N)$.
If each agent's feasible region is independent of other agents' decisions,
that is, $K_i(x_{-i}) \equiv K_i$ for some nonempty set $K_i$, then the
problem is called a Nash equilibrium problem (NEP).

In addition to the GNEP or NEP setting, if we have an agent formulating some
equilibrium conditions, such as market clearing conditions, as a variational
inequality (\VI{}) whose definition is given in
Section~\ref{sec:preliminaries}, we call the problem multiple optimization
problems with equilibrium constraints (MOPEC).
For example, $x^*=(x_1^*,\dots,x_N^*,x_{N+1}^*)$ is a solution to the MOPEC if
it satisfies
\begin{equation}
  \begin{aligned}
    x_i^* &\in \argmin_{x_i \in K_i(x^*_{-i}) \subset \real^{n_i}} \quad
    f_i(x_i,x^*_{-i}), \quad \text{for} \quad i=1,\dots,N,\\
    x_{N+1}^* &\in \SOL(K_{N+1}(x_{-(N+1)}^*),F(\cdot,x_{-(N+1)}^*)),
  \end{aligned}
  \tag{MOPEC}
  \label{eq:mopec}
\end{equation}
where $\SOL(K,F)$ denotes the solution set of a variational inequality
$\VI(K,F)$, assuming that
$K_{N+1}(x_{-(N+1)})$ is a nonempty closed convex set for each given
$x_{-(N+1)}$ and $F(x)$ is a continuous function.
We call agent $i$ for $i=1,\dots,N$ an optimization agent and agent $(N+1)$
an equilibrium agent.

Solutions of these problems using modeling languages are usually obtained by
transforming the problem into their equivalent complementarity form, such as
a mixed complementarity problem (\MCP{}), and then solving the complementarity
problem using a specialized solver, for example \PATH{}~\cite{dirkse95}.
This implies that users need to compute the Karush-Kuhn-Tucker (KKT)
conditions of each optimization agent by hand and then manually specify the
complementarity relationships within modeling
languages~\cite{ferris99-siam,rutherford95}.
Similar transformations are needed to formulate equilibrium problems in
their variational forms represented by QVIs as we show in
Section~\ref{sec:preliminaries}.

This approach has several drawbacks.
It is time-consuming and error-prone because of the derivative computation.
The problem structure becomes lost once it is converted into the
complementarity form: it is difficult to tell what the original model is and
which agent controls what variables and equations (functions and constraints)
by just reading the complementarity form.
For QVI formulations, we lose the information about what variables are used as
parameters to define the feasible region.
All variables and equations are endogenous in that form.
This may restrict opportunities for back-end solvers to detect and make
full use of the problem structure.
Modifying the model such as adding/removing variables and equations may not
be easy: it typically involves a lot of derivative recomputation.

For more intuitive and efficient equilibrium programming, that is, formulating
GNEP, MOPEC, or QVI in modeling languages, the paper~\cite{ferris09} briefly
mentioned that the EMP framework can be used to specify GNEPs and MOPECs.
The goal of EMP is to enable users to focus on the problem description itself
rather than spending time and checking errors on the complementarity form derivation.
Users annotate variables and equations of the problem in a similar way to
its algebraic representation and write them into an \texttt{empinfo} file.
The modeling language reads that file to identify high level structure of
the problem such as the number of agents and agent's ownership of variables
and equations.
It then automatically constructs the corresponding complementarity form and
solves it using complementarity solvers.
However, neither detailed explanations about its underlying assumptions and
how to use it are given, nor are the QVI formulations considered
in~\cite{ferris09}.

In this paper, we present detailed explanation of the existing EMP framework
for equilibrium programming for the first time.
We also describe its extensions to incorporate some new sophisticated
structures, such as {\it shared constraints}, {\it shared variables}, and
{\it QVI formulations}, and their implications with examples from the
literature.
Our extensions allow a natural translation of the algebraic formulation
into modeling languages while capturing high level structure of the problem
so that the back-end solver can harness the structure for improved
performance.

Specifically, our framework allows shared constraints to be represented
without any replications and makes it easy to switch between different
solution types associated with them, for example variational
equilibrium~\cite[Definition 3.10]{facchinei10}.
We introduce a new concept, shared variables, and show their manifestations in
the literature.
Shared variables have potential for many different uses:
i) they can be used to reduce the density of the model;
ii) they can model some EPECs sharing the same variables and constraints to
represent equilibrium constraints;
iii) we can easily switch between price-taking and price-making agents in
economics models;
iv) they can be used to model shared objective functions.
The last case opens the door for our framework to be used to model the block
coordinate descent method, where agents now correspond to a block of
variables.
Finally, we define a new construct that allows QVI formulations to be
specified in an intuitive and natural way.
The new features have been implemented and are available within GAMS/EMP.
In this case, we use a problem reformulation solver JAMS, and choose
formulations if necessary in an option file {\tt jams.opt}.

The rest of the paper is organized as follows.
In Section~\ref{sec:preliminaries}, we define the equilibrium formulations our
framework allows.
We describe conditions under which the complementarity form is equivalent to
the given equilibrium problem and its variational form.
Section~\ref{sec:equilibrium} presents the underlying assumptions of the
existing framework and shows how we can model equilibrium problems satisfying
these assumptions.
In Sections~\ref{sec:shared-constr}-\ref{sec:impl}, we present sophisticated
structures that violate the assumptions and introduce our modifications to incorporate them into our framework.
Section~\ref{sec:shared-constr} describes shared constraints and
presents a new construct to define the type of solutions, either GNEP
equilibria or variational equilibria, associated with them.
In Section~\ref{sec:impl}, we introduce shared variables and various
complementarity formulations for them.
Section~\ref{sec:qvi} presents a new construct to specify QVIs and compares
two equivalent ways of specifying equilibrium problems in either GNEP or
QVI form.
At the end of each section of Sections~\ref{sec:equilibrium}-\ref{sec:qvi},
we provide examples from the literature that can be neatly formulated using
the feature of our framework.
Section~\ref{sec:conclusions} concludes the paper, pointing out some areas for
future extensions.

\section{Preliminaries}
\label{sec:preliminaries}

For given equilibrium problems or their variational forms, the default action
of our framework converts them into MCPs and computes a solution to those
complementarity problems.
In this section, we describe equivalences of the equilibrium problems with
quasi-variational inequalities, variational inequalities, and mixed
complementarity problems.

We first introduce QVIs, VIs, and MCPs in a finite-dimensional space.
For a given continuous function $F:\real^n \rightarrow \real^n$ and a
point-to-set mapping $K(x):\real^n \rightrightarrows \real^n$ where $K(x)$
is a closed convex set for each $x \in \real^n$, $x^* \in K(x^*)$ is a
solution to the $\QVI(K,F)$ if
\begin{equation}
  \langle F(x^*), x - x^* \rangle \ge 0, \quad \forall x \in K(x^*),
  \tag{QVI}
  \label{eq:qvi}
\end{equation}
where $\langle \cdot,\cdot \rangle$ is the Euclidean inner product.

If we restrict the point-to-set mapping $K(\cdot)$ to be a fixed closed
convex set $K \subset \real^n$, then $x^* \in K$ is a solution to the
$\VI(K,F)$ if
\begin{equation}
  \langle F(x^*), x - x^* \rangle \ge 0, \quad \forall x \in K.
  \tag{VI}
  \label{eq:vi}
\end{equation}

If we further specialize to the case where the feasible region is a box
$B=\{x \in \real^n \mid l_i \le x_i \le u_i, \text{ for }i=1,\dots,n\}$ with
$l_i \le u_i$ and $l_i \in \real \cup \{-\infty\}$ and
$u_i \in \real \cup \{\infty\}$,
the $\VI(B,F)$ is typically termed a mixed complementary problem.
In this case, $x^* \in B$ is a solution to the $\MCP(B,F)$ if one of
the following conditions holds for each $i=1,\dots,n$:
\begin{equation}
  \begin{aligned}
    x_i^* = l_i, && F_i(x^*) \ge 0,\\
    l_i \le x_i^* \le u_i, && F_i(x^*) = 0,\\
    x_i^* = u_i, && F_i(x^*) \le 0.
  \end{aligned}
  \tag{MCP}
  \label{eq:mcp}
\end{equation}
In shorthand notation, the above condition is written as $l \le x^* \le u
\, \perp \, F(x^*)$.
We sometimes put a bound constraint on a function explicitly when the
corresponding variable has only one-sided bound and use $\MCP(x,F)$ when the
feasible region of $x$ is clear from the context.

Throughout this paper, we assume by default that equilibrium problems are of
the form~\eqref{eq:mopec}, and there are $(N+1)$ number of agents where the
first $N$ agents are optimization agents, and the $(N+1)$th agent is an
equilibrium agent.
When there is no equilibrium agent, then the problem becomes a (generalized)
Nash equilibrium problem.
If there are no optimization agents but a single equilibrium agent, then the
problem is a variational inequality.
All results in this section hold in the case where either type of agent
is not present.

The results described below are simple extensions of the existing
results found in~\cite{harker91}.
We first show the equivalence between the equilibrium problems and their
associated \QVI{}s.

\begin{proposition}
  If $f_i(\cdot,\cdot)$ is continuously differentiable,
  $f_i(\cdot,x_{-i})$ is a convex function, and $K_i(x_{-i})$ is a
  closed convex set for each given $x_{-i}$, then $x^*$ is a solution to the
  equilibrium problem~\eqref{eq:mopec} if and only if it is a solution to
  the $\QVI(K,F)$ where
  \begin{equation*}
    \begin{aligned}
      K(x) &= \prod_{i=1}^{N+1} K_i(x_{-i}),\\
      F(x) &= (\nabla_{x_1}f_i(x_1,x_{-1})^\tr,
      \dots,\nabla_{x_N}f_N(x_N,x_{-N})^\tr,
      G(x_{N+1},x_{-(N+1)})^\tr)^\tr.
    \end{aligned}
  \end{equation*}
  \label{prop:equilprob-qvi}
\end{proposition}
\begin{proof}
  $(\Rightarrow)$ Let $x^*$ be a solution to~\eqref{eq:mopec}.
  For optimization agents, the first-order optimality conditions are necessary
  and sufficient by the given assumption.
  Therefore we have
  \begin{equation*}
    \langle \nabla_{x_i}f_i(x_i^*,x_{-i}^*), x_i - x_i^* \rangle \ge 0,
    \quad \forall x_i \in K_i(x_{-i}^*), \text{ for }i=1,\dots,N.
  \end{equation*}
  Also we have
  \begin{equation*}
    \langle G(x_{N+1}^*,x_{-(N+1)}^*), x_{N+1} - x_{N+1}^* \rangle \ge
    0,
    \quad \forall x_{N+1} \in K_{N+1}(x_{-(N+1)}^*).
  \end{equation*}
  The result follows.

  $(\Leftarrow)$
  Let $x^*$ be a solution to the $\QVI(K,F)$.
  The result immediately follows from the fact that $K(x)$ is a product space
  of $K_i(x_{-i})$'s for $i=1,\dots,N+1$.\qed
\end{proof}

If each agent $i$ has knowledge of a closed convex set $X$ and
use this to define its feasible region $K_i(x_{-i})$ using a shared
constraint $K_i(x_{-i}) := \{x_i \in \real^{n_i} \mid (x_i,x_{-i}) \in X\}$,
then the $\QVI(K,F)$ can be solved using a simpler $\VI(X,F)$.
\begin{proposition}
  Suppose that $K_i(x_{-i})=\{x_i \in \real^{n_i} \mid (x_i,x_{-i}) \in X\}$
  for $i=1,\dots,N+1$ with $X$ being a closed convex set.
  If $x^*$ is a solution to the $\VI(X,F)$ with $F$ defined in
  Proposition~\ref{prop:equilprob-qvi}, then it is a solution to the
  $\QVI(K,F)$, thus it is a solution to~\eqref{eq:mopec} with the
  same assumptions on $f_i(\cdot)$ given in
  Proposition~\ref{prop:equilprob-qvi}.
  The converse may not hold.
  \label{prop:equilprob-vi}
\end{proposition}
\begin{proof}
  $(\Rightarrow)$
  Let $x^*$ be a solution to the $\VI(X,F)$.
  Clearly, $x^* \in K(x^*)$.
  We prove by contradiction.
  Suppose there exists $x \in K(x^*)$ such that $\langle F(x^*), x - x^*
  \rangle < 0$.
  There must exist $i \in \{1,\dots,N+1\}$ satisfying
  $\langle F_i(x^*), x_i - x_i^* \rangle < 0$.
  Set $\tilde{x} = (x_i,x_{-i}^*)$.
  As $x_i \in K_i(x_{-i}^*)$, $\tilde{x} \in X$.
  Then, $\langle F(x^*), \tilde{x} - x^* \rangle < 0$, which is a
  contradiction.

  $(\nLeftarrow)$ See the example in Section 3 of~\cite{harker91}.\qed
\end{proof}

When the constraints are explicitly given as equalities and inequalities with
a suitable constraint qualification holding, we can compute a solution to
the equilibrium problems using their associated \MCP{} and vice versa.
Throughout this section, by a suitable constraint qualification we mean a
constraint qualification implying the KKT conditions hold at a local
optimal solution, for example the Mangasarian-Fromovitz or the
Slater constraint qualification.
Also when we say a constraint qualification holds at $x$, we imply that it
holds at $x_i \in K_i(x_{-i})$ for each agent $i$.

\begin{proposition}
  Suppose that $K_i(x_{-i})=\{x_i \in [l_i,u_i] \mid
  h_i(x_i,x_{-i})=0, g_i(x_i,x_{-i}) \le 0\}$ where
  $h_i(\cdot):\real^n \rightarrow \real^{v_i}$ is an affine function,
  each $g_i(\cdot):\real^n \rightarrow \real^{m_i}$
  is continuously differentiable and a convex function of $x_i$ and
  $l_i \le u_i, l_i \in \real^{n_i} \cup \{-\infty\}^{n_i}$, and
  $u_i \in \real^{n_i} \cup \{\infty\}^{n_i}$.
  With the same assumptions on $f_i$ given in
  Proposition~\ref{prop:equilprob-qvi}, $x^*$ is a solution
  to~\eqref{eq:mopec} if and only if $(x^*,\lambda^*,\mu^*)$ is a solution
  to the $\MCP(B,F)$, assuming that a suitable constraint qualification holds
  at $x^*$ with
  \begin{equation*}
    \begin{aligned}
      B &= &&\prod_{i=1}^{N+1} [l_i,u_i] \times \real^v \times \real^m_-,
      \quad v = \sum_{i=1}^{N+1}v_i, \quad m = \sum_{i=1}^{N+1} m_i,\\
      F(x,\lambda,\mu) &= &&(
      (\nabla_{x_1}f_1(x) - \nabla_{x_1}h_1(x)\lambda_1 -
      \nabla_{x_1}g_1(x)\mu_1)^\tr,\dots,\\
      &&&\,\,(\nabla_{x_N}f_N(x) - \nabla_{x_N}h_N(x)\lambda_N -
      \nabla_{x_N}g_N(x)\mu_N)^\tr,\\
      &&&\,\,(G(x) - \nabla_{x_{N+1}}h_{N+1}(x)\lambda_{N+1} -
      \nabla_{x_{N+1}}g_{N+1}(x)\mu_{N+1})^\tr,\\
      &&&h_1(x)^\tr,\dots,h_{N+1}(x)^\tr,\\
      &&&g_1(x)^\tr,\dots,g_{N+1}(x)^\tr)^\tr.
    \end{aligned}
  \end{equation*}
  \label{prop:equilprob-mcp}
\end{proposition}
\begin{proof}
  $(\Rightarrow)$
  Let $x^*$ be a solution to~\eqref{eq:mopec}.
  Using the KKT conditions of each optimization agent and the \VI{}, and
  constraint qualification at $x^*$, there exist $(\lambda^*,\mu^*)$ such that
  \begin{equation}
    \begin{aligned}
      \nabla_{x_i}f_i(x^*) - \nabla_{x_i}h_i(x^*)\lambda_i^* -
      \nabla_{x_i}g_i(x^*)\mu_i^* && \perp &&&
      l_i \le x_i^* \le u_i, && \text{for} \quad i=1,\dots,N,\\
      G(x^*) - \nabla_{x_i}h_i(x^*)\lambda_i^* -
      \nabla_{x_i}g_i(x^*)\mu_i^* && \perp &&&
      l_i \le x_i^* \le u_i, && \text{for} \quad i=N+1,\\
      0 = h_i(x^*) && \perp &&& \lambda_i^* \text{ free},
      && \text{for} \quad i=1,\dots,N+1,\\
      0 \ge g_i(x^*) && \perp &&&
      \mu_i^* \le 0, && \text{for} \quad i=1,\dots,N+1.\\
    \end{aligned}
    \label{eq:equilprob-mcp}
  \end{equation}
  Thus $(x^*,\lambda^*,\mu^*)$ is a solution to the $\MCP(B,F)$.

  $(\Leftarrow)$
  Let $(x^*,\lambda^*,\mu^*)$ be a solution to the $\MCP(B,F)$.
  Then $(x^*,\lambda^*,\mu^*)$ satisfies~\eqref{eq:equilprob-mcp}.
  Since the constraint qualification holds at $x^*$, we have
  $N_{K_i(x_{-i}^*)} =
  \{-\nabla_{x_i}h_i(x^*)\lambda_i-\nabla_{x_i}g_i(x^*)\mu_i \mid
  0 = h_i(x^*) \perp \lambda_i,
  0 \ge g_i(x^*) \perp \mu_i \le 0\} + N_{[l_i,u_i]}(x_i^*)$
  for $i=1,\dots,N+1$.
  The result follows from convexity.\qed
\end{proof}

If the convexity assumptions on the objective functions and the constraints of
optimization agents' problems do not hold, then one can easily check that a
stationary point to~\eqref{eq:mopec} is a solution to the \MCP{} model defined
in Proposition~\ref{prop:equilprob-mcp} and vice versa.
By a stationary point, we mean that $x_i^*$ satisfies the first-order
optimality conditions of each optimization agent $i$'s problem, and
$x_{N+1}^*$ is a solution to the equilibrium agent's problem.

Finally, we present the equivalence between QVIs and MCPs.
\begin{proposition}
  For a given $\QVI(K,F)$, suppose that $K(x)=\{l \le y \le u \mid
  h(y,x)=0,g(y,x) \le 0\}$ where $h:\mathbb{R}^{n \times n} \rightarrow
  \mathbb{R}^v$ and $g:\mathbb{R}^{n \times n} \rightarrow \mathbb{R}^m$.
  Assuming that a suitable constraint qualification holds, $x^*$ is a
  solution to the $\QVI(K,F)$ if and only if $(x^*,\lambda^*,\mu^*)$ is a
  solution to the $\MCP(B,\tilde{F})$ where
  \begin{equation*}
    \begin{aligned}
      B &= && [l,u] \times \real^v \times \real^m_-,\\
      \tilde{F}(x,\lambda,\mu) &= &&
      \begin{bmatrix}
        F(x) - \nabla_{y}h(x,x)\lambda - \nabla_{y}g(x,x)\mu\\
        h(x,x)\\
        g(x,x)
      \end{bmatrix}
    \end{aligned}
  \end{equation*}
  \label{prop:qvi-mcp}
\end{proposition}
\begin{proof}
  By applying similar techniques used in the proof of
  Proposition~\ref{prop:equilprob-mcp}, we get the desired result.\qed
\end{proof}

For a given equilibrium problem or quasi-variational inequality our
framework generates the \MCP{} model defined in
Propositions~\ref{prop:equilprob-mcp}-\ref{prop:qvi-mcp}, respectively, and
solves it using \PATH{}.
If the feasible region is defined by a shared constraint, users can choose
between the \VI{} defined in Proposition~\ref{prop:equilprob-vi} and the
\MCP{} by specifying the solution type.
This will be discussed in Section~\ref{sec:shared-constr}.
Other extensions are found in Sections~\ref{sec:impl}-\ref{sec:qvi}.
While this constitutes one method of solution, the ability to define the
structured equilibria explicitly opens the door for new solution
methods~\cite{youngdae17}.

\section{Modeling equilibrium problems using the existing EMP framework}
\label{sec:equilibrium}

We now describe how to specify equilibrium problems in modeling languages
using the EMP framework.
While this is implemented in GAMS syntax, the extension to other modeling
systems is straightforward.
We first present the underlying assumptions on the specification and discuss
their limitations in Section~\ref{subsec:underlying-assumptions}.
Examples from the literature are given in Section~\ref{subsec:examples-existing}.
In Sections~\ref{sec:shared-constr}-\ref{sec:impl}, we relax these
assumptions to take more sophisticated structures into account.

\subsection{Specifying equilibrium problems and underlying assumptions}
\label{subsec:underlying-assumptions}

Standard equilibrium problems can be specified in modeling languages using
our framework.
Suppose that we are given the following NEP:
\begin{equation}
  \begin{aligned}
    &\text{find} && (x_1^*,\dots,x_N^*) && \text{satisfying},\\
    &x_i^* \in &&\argmin_{x_i} && f_i(x_i,x_{-i}^*),\\
    &&& \text{subject to} && h_i(x_i) = 0,\\
    &&& && g_i(x_i) \le 0, \text{ for }i=1,\dots,N.
  \end{aligned}
  \label{eq:nep}
\end{equation}

We need to specify each agent's variables, its objective function, and constraints.
Functions and constraints are given as a closed-form in modeling languages:
they are explicitly written using combinations of mathematical operators such
as summation, multiplication, square root, log, and so on.
The EMP partitions the variables, functions, and constraints among the
agents using annotations given in an {\tt empinfo} file.
For example, we may formulate and solve~\eqref{eq:nep} within GAMS/EMP as
follows:
\begin{lstlisting}[caption={Modeling the
NEP},label={lst:nep-emp},language=GAMS]
variables obj(i), x(i);
equations deff(i), defh(i), defg(i);

* Definitions of deff(i), defh(i), and defg(i) are omitted for expository purposes.

model nep / deff, defh, defg /;

file empinfo / '%emp.info%' /;
put empinfo 'equilibrium' /;
loop(i,
  put 'min', obj(i), x(i), deff(i), defh(i), defg(i) /;
);
putclose;

solve nep using emp;
\end{lstlisting}

Let us explain Listing~\ref{lst:nep-emp}. Variable \texttt{obj(i)} holds the
value of $f_i(x)$, \texttt{x(i)} represents variable $x_i$, and
\texttt{deff(i)}, \texttt{defh(i)}, and \texttt{defg(i)} are the closed-form
definitions of the objective function $f_i(x)$ and the constraints $h_i(x_i)$
and $g_i(x_i)$, respectively, for $i=1,\dots,N$.
Equations listed in the model statement and variables in these equations
constitute the model \texttt{nep}.

Once the model is defined, a separate \texttt{empinfo} file is created to
specify the equilibrium problem.
In the above case, the \texttt{empinfo} file has the following contents:

\begin{center}
\begin{BVerbatim}
equilibrium
min obj('1') x('1') deff('1') defh('1') defg('1')
...
min obj('N') x('N') deff('N') defh('1') defg('N')
\end{BVerbatim}
\end{center}

The \texttt{equilibrium} keyword informs EMP that the annotations are for an
equilibrium problem.
A list of agents' problem definitions separated by either a \texttt{min} or
\texttt{max} keyword for each optimization agent follows.
For each \texttt{min} or \texttt{max} keyword, the objective variable to
optimize and a list of agent's decision variables are given.
After these variables, a list of equations that define the agent's objective
function and constraints follows.
We say that variables and equations listed are owned by the agent.
Note that variables other than \texttt{x('1')} that appear in
\texttt{deff('1')}, \texttt{defh('1')}, or \texttt{defg('1')} are treated as
parameters to the first agent's problem; that is how we define $x_{-i}$.
The way each agent's problem is specified closely resembles its algebraic
formulation~\eqref{eq:nep}, and our framework reconstructs each agent's
problem by reading the {\tt empinfo} file.

The framework does not require any special keyword to distinguish between a
NEP and a GNEP.
If the function $h_i$ or $g_i$ is defined using other agents' decisions, that
is, $h_i(x_i,x_{-i}) = 0$ or $g_i(x_i,x_{-i}) \le 0$, the equilibrium model
written in Listing~\ref{lst:nep-emp} becomes a GNEP.
The distinction between the NEP and the GNEP depends only on how the
constraints are defined.

Note that in the {\tt empinfo} file above, each variable and equation
is owned exclusively by a single agent.
There is no unassigned variable or equation.
In the standard framework, neither multiple ownership nor missing ownership
are allowed; otherwise an error is generated.
Formally, the standard framework assumes the following:

\begin{assumption}
  A model of an equilibrium problem described by equations and variables is
  assumed to have the following properties in the empinfo file:
  \begin{itemize}
  \item Each equation of the model is owned by a single agent.
  \item Each variable of the model is owned by a single agent.
  \end{itemize}
  \label{asm:original-emp}
\end{assumption}

An implication of Assumption~\ref{asm:original-emp} is that the current
framework does not allow \textit{shared objective functions}, \textit{shared
  constraints}, and \textit{shared variables}.
Sections~\ref{sec:shared-constr}-\ref{sec:impl} give examples of problems
that violate Assumption~\ref{asm:original-emp} and provide techniques to
overcome or relax the requirements.

The MOPEC model can be defined in a very similar way. Suppose that we are
given the following MOPEC:
\begin{equation}
  \begin{aligned}
    &\text{find} && (x_1^*,\dots,x_N^*,p^*) \quad \text{satisfying},\\
    &x_i^* \in &&\argmin_{x_i} \quad f_i(x_i,x_{-i}^*),\\
    &&& \text{subject to} \quad h_i(x_i,x_{-i}^*) = 0,\\
    &&& \quad\quad\quad\quad\,\,\,\,\,\,\,
    g_i(x_i,x_{-i}^*) \le 0, \quad\text{for}\quad i=1,\dots,N,\\
    &p^* \in &&\SOL(K(x^*), V(p,x^*)),\\
    &&& \text{where} \quad K(x^*) := \{ p \mid w(p,x^*) \le 0\}.
  \end{aligned}
  \label{eq:gams-mopec}
\end{equation}
Assuming that $p \in \mathbb{R}^r$, we can then
formulate~\eqref{eq:gams-mopec} within GAMS/EMP in the following way:
\begin{lstlisting}[caption={Modeling the MOPEC},label={lst:mopec-emp},language=GAMS]
variables obj(i), x(i), p(j);
equations deff(i), defh(i), defg(i), defV(j), defw;

model mopec / deff, defh, defg, defV, defw /;

file empinfo / '%emp.info%' /;
put empinfo 'equilibrium' /;
loop(i,
  put 'min', obj(i), x(i), deff(i), defh(i), defg(i) /;
);
put 'vi defV p defw' /;
putclose empinfo;
\end{lstlisting}

In addition to optimization agents, we now have an equilibrium agent defined
with the \texttt{'vi'} keyword in Listing~\ref{lst:mopec-emp}.
The \texttt{'vi'} keyword is followed by variables, function-variable pairs,
and constraints.
Functions paired with variables constitute a \VI{} function, and the order of
functions and variables appeared in the pair is used to determine which
variable is assigned to which function when we compute the inner product in
the~\eqref{eq:vi} definition.
In this case, we say that each \VI{} function is matched with each variable
having the same order in the pair, i.e., \texttt{defV(j)} is matched with
\texttt{p(j)} for each $j=1,\dots,r$.
After all matching information is described, constraints follow.
Hence, the \VI{} function is \texttt{defV}, its variable is \texttt{p}, and
\texttt{defw} is a constraint.
The functions $f_i$, $h_i$, and $g_i$, defined in \texttt{deff(i)},
\texttt{defh(i)}, and \texttt{defg(i)} equations, respectively, may now
include the variable $p$.
One can easily verify that the specification in the \texttt{empinfo} file
satisfies Assumption~\ref{asm:original-emp}.

Variables, that are used only to define the constraint set and are owned by
the \VI{} agent, must be specified before any explicit function-variable
pairs.
In this case, we call those variables {\it preceding variables}.
The interface automatically assigns them to a zero function, that is, a
constant function having zero value.
For example, if we specify \texttt{`vi z Fy y'} where \texttt{z} and
\texttt{y} are variables and \texttt{Fy} is a \VI{} function matched with
\texttt{y}, then \texttt{z} is a preceding variable.
In this case, our interface automatically creates an artificial function
\texttt{Fz} defined by $F(z) \equiv 0$ and matches it with variable
\texttt{z}.

\subsection{Examples}
\label{subsec:examples-existing}

Examples of NEP, GNEP, and MOPEC taken from the literature are formulated in
the following sections using the EMP framework.

\subsubsection{NEP}
\label{subsubsec:exp-nash}

We consider the following oligopolistic market equilibrium problem~\cite{harker84,murphy82}:
\begin{equation}
  \begin{aligned}
    &\text{find} && (q_1^*,\dots,q_5^*) && \text{satisfying},\\
    &q_i^* \in && \argmax_{q_i \ge 0} && q_ip\left(\sum_{j=1,j\neq
        i}^5 q^*_j+q_i\right)  - f_i(q_i),\\
    &&&\text{where} && p(Q) := 5000^{1/1.1}(Q)^{-1/1.1},\\
    &&&             && f_i(q_i) := c_iq_i +
    \frac{\beta_i}{\beta_i+1}K_i^{-1/\beta_i}q_i^{(\beta_i+1)/\beta_i},\\
    &&&&& (c_i,K_i,\beta_i) \text{ is problem data},
    && \text{for} \quad i=1,\dots,5.\\
  \end{aligned}
  \label{eq:exp-nash}
\end{equation}

There are five firms, and each firm provides a homogeneous product with amount
$q_i$ to the market while trying to maximize its profit in a noncooperative
way.
The function $p(\cdot)$ is the inverse demand function, and its
value is determined by the sum of the products provided by all the firms.
The function $f_i(\cdot)$ is the total cost of firm $i$.
The problem~\eqref{eq:exp-nash} is a NEP.

Listing~\ref{lst:exp-nash} shows an implementation of~\eqref{eq:exp-nash}
within GAMS/EMP.
As we see, the \texttt{empinfo} file is a natural translation of the algebraic
form of~\eqref{eq:exp-nash}.
Using the same starting value as in~\cite{harker84,murphy82}, our GAMS/EMP
implementation computed a solution $q^*=(36.933,41.818,43.707,42.659,39.179)^\tr$ that
is consistent with the one reported in those papers.

\begin{lstlisting}[caption={Implementation of the NEP~\eqref{eq:exp-nash}
within GAMS/EMP},label={lst:exp-nash},language=GAMS]
sets i agents / 1*5 /;
alias(i,j);

parameters c(i)    / 1  10, 2   8, 3   6, 4   4, 5   2 /,
           K(i)    / 1   5, 2   5, 3   5, 4   5, 5   5 /,
           beta(i) / 1 1.2, 2 1.1, 3 1.0, 4 0.9, 5 0.8 /;

variables obj(i);
positive variables q(i);

equations objdef(i);

objdef(i)..
    obj(i) =e= q(i)*5000**(1.0/1.1)*sum(j, q(j))**(-1.0/1.1) - (c(i)*q(i) + beta(i)/(beta(i)+1)*K(i)**(-1/beta(i))*q(i)**((beta(i)+1)/beta(i)));

model nep / objdef /;

file empinfo / '%emp.info%' /;
put empinfo 'equilibrium' /;
loop(i,
    put 'max', obj(i), q(i), objdef(i) /;
);
putclose empinfo;

q.l(i) = 10;
solve nep using emp;
\end{lstlisting}

\subsubsection{GNEP}
\label{subsubsec:exp-gnep}

We use the following GNEP example derived from the QVI example
of~\cite[page 14]{outrata95}:
\begin{equation}
  \begin{aligned}
    &\text{find} && (x_1^*,x_2^*) && \text{satisfying},\\
    &x_1^* \in && \argmin_{0 \le x_1 \le 11} && x_1^2 + \frac{8}{3}x_1x^*_2 -
    \frac{100}{3}x_1,\\
    &&&\text{subject to} && x_1 + x^*_2 \le 15,\\
    &x_2^* \in && \argmin_{0 \le x_2 \le 11} && x_2^2 + \frac{5}{4}x^*_1x_2 -
    22.5x_2,\\
    &&&\text{subject to} && x^*_1 + x_2 \le 20.
  \end{aligned}
  \label{eq:exp-gnep}
\end{equation}

In~\eqref{eq:exp-gnep}, each agent solves a strongly convex optimization
problem.
Not only the objective functions but also the feasible region of each agent is
affected by other agent's decision.
Hence it is a GNEP.
Listing~\ref{lst:exp-gnep} shows an implementation of~\eqref{eq:exp-gnep}
within GAMS/EMP.
Our model has computed a solution
$(x_1^*,x_2^*)=(10,5)$ that is consistent with the one reported in~\cite{outrata95}.
In Section~\ref{subsec:qvi-example}, we show that~\eqref{eq:exp-gnep} can be
equivalently formulated as a QVI using our extension to the EMP framework.

\begin{lstlisting}[caption={Implementation of the GNEP~\eqref{eq:exp-gnep}
within GAMS/EMP},label={lst:exp-gnep},language=GAMS]
set i / 1*2 /;
alias(i,j);

variable obj(i);
positive variable x(i);

equation defobj(i), cons(i);

defobj(i)..
    obj(i) =E=
    (sqr(x(i)) + 8/3*x(i)*x('2') - 100/3*x(i))$(i.val eq 1) +
    (sqr(x(i)) + 5/4*x('1')*x(i) - 22.5*x(i))$(i.val eq 2);

cons(i)..
    sum(j, x(j)) =L= 15$(i.val eq 1) + 20$(i.val eq 2);

x.up(i) = 11;

model gnep / defobj, cons /;

file empinfo / '%emp.info%' /;
put empinfo 'equilibrium' /;
loop(i,
    put 'min', obj(i), x(i), defobj(i), cons(i) /;
);
putclose empinfo;

solve gnep using emp;
\end{lstlisting}

\subsubsection{MOPEC}
\label{subsubsec:exp-mopec}

We present a general equilibrium example in economics~\cite[Section
3]{mathiesen87} and model it as a MOPEC.
While~\cite{mathiesen87} formulated the problem as a
complementarity problem by using the closed form of the utility maximizing
demand function, we formulate it as a MOPEC by explicitly introducing a
utility-maximizing optimization agent (the consumer) to compute the demand.

Let us briefly explain the general equilibrium problem we consider.
We use the notations and explanation from~\cite{mathiesen87}.
There are three types of agents:
i) profit-maximizing producers;
ii) utility-maximizing consumers;
iii) a market determining the price of commodities based on production and
demand.
The problem is given with a technology matrix $A$, an initial endowment $b$,
and the demand function $d(p)$.
The coefficient $a_{ij} > 0$ (or $a_{ij} < 0)$ of $A$ indicates output (or
input) of commodity $i$ for each unit activity of producer $j$.
For a given price $p$, $d(p)$ is the demand of consumers maximizing their
utilities within their budgets, where budgets depend on the price $p$ and
initial endowment $b$.
Assuming that $y, x$, and $p$ represent activity of producers, demands of
consumers, and prices of commodities, respectively, we say that
$(y^*,x^*,p^*)$ is a general equilibrium if it satisfies the following:
\begin{equation}
  \begin{aligned}
    & \text{No positive profit for each activity} &&& -A^\tr p^* \ge 0,\\
    & \text{No excess demand}   &&& b + Ay^* - x^* \ge 0,\\
    & \text{Nonnegativity}      &&& p^* \ge 0, y^* \ge 0,\\
    & \text{No activity for earning negative profit}
    &&& (-A^\tr p^*)^\tr y^*=0,\\
    & \text{and positive activity implies balanced profit} ,\\
    & \text{Zero price for excess supply} &&& p^{*\tr}(b + Ay^* - x^*) =
    0,\\
    & \text{and market clearance for positive price,}\\
    & \text{Utility maximizing demand} &&& x^* \in \argmax_{x}
    \quad \text{utility}(x),\\
    &&&&\quad\quad \text{subject to} \quad p^{*\tr} x \le p^{*\tr} b.
  \end{aligned}
\end{equation}

We consider a market where there are a single producer, a single consumer, and
three commodities.
To compute the demand function without using its closed form, we introduce a
utility-maximizing consumer explicitly in the model.
Our GAMS/EMP model finds a solution $y^*=3, x^*=(3,2,0)^\tr, p^*=(6,1,5)^\tr$
for $\alpha=0.9$ that is consistent with the one in~\cite{mathiesen87}.

\begin{lstlisting}[caption={Implementation of the MOPEC within
GAMS/EMP},label={lst:exp-mopec},language=GAMS]
set i commodities / 1*3 /;

parameters ATmat(i)  technology matrix  / 1 1  , 2 -1 , 3 -1 /,
           s(i)      budget share       / 1 0.9, 2 0.1, 3 0 /,
           b(i)      endowment          / 1 0  , 2 5  , 3 3 /;

variable u    utility of the consumer;
positive variables y         activity of the producer,
                   x(i)      Marshallian demand of the consumer,
                   p(i)      prices;

equations mkt(i)    constraint on excess demand,
          profit    profit of activity,
          udef      Cobb-Douglas utility function,
          budget    budget constraint;

mkt(i)..
    b(i) + ATmat(i)*y - x(i) =G= 0;

profit..
    sum(i, -ATmat(i)*p(i)) =G= 0;

udef..
    u =E= sum(i, s(i)*log(x(i)));

budget..
    sum(i, p(i)*x(i)) =L= sum(i, p(i)*b(i));

model mopec / mkt, profit, udef, budget /;

file empinfo / '%emp.info%' /;
put empinfo 'equilibrium' /;
put 'max', u, 'x', udef, budget /;
* We have mkt perp p and profit perp y, the fourth and fifth conditions of (6).
put 'vi mkt p profit y' /;
putclose empinfo;

* The second commodity is used as a numeraire.
p.fx('2') = 1;
x.l(i) = 1;

solve mopec using emp;
\end{lstlisting}

\section{Modeling equilibrium problems with shared constraints}
\label{sec:shared-constr}

This section describes our first extension to model shared constraints and
to compute different types of solutions associated with them.

\subsection{Shared constraints and limitations of the existing framework}
\label{subsec:shared-constr-def}

We first define shared constraints in equilibrium problems, specifically when
they are explicitly given as equalities or inequalities.
\begin{definition}
  In equilibrium problems, if the same constraint, given explicitly as
  an equality or an inequality, appears multiple times in different agents'
  problem definitions, then it is a shared constraint.
  \label{def:shared-constr}
\end{definition}

For example, a constraint $h(x) \le 0$ (with no subscript $i$ on $h$) is a
shared constraint in the following GNEP:
\begin{example}
  Find $(x_1^*,\dots,x_N^*)$ satisfying
  \begin{equation*}
    \begin{aligned}
      &x_i^* \in &&\argmin_{x_i} && f_i(x_i,x_{-i}^*),\\
      &&&\text{ subject to } && g_i(x_i,x_{-i}^*) \le 0,\\
      &&&&&h(x_i,x_{-i}^*) \le 0, &&\text{for} \quad i=1,\dots,N.\\
    \end{aligned}
  \end{equation*}
  \label{ex:shared-constr}
\end{example}

Our definition of a shared constraint allows each agent's feasible region to
be defined with a combination of shared and non-shared constraints.
Our definition subsumes the cases in~\cite{facchinei07,facchinei10}, where
each agent's feasible region is defined by the shared constraint only: in that
situation there are no $g_i(x)$'s.
In our framework, the shared constraint can also be defined over some subset
of agents.
For expository ease throughout this section, we use Example
\ref{ex:shared-constr}, but the extension to the more general setting is
straightforward.

Shared constraints are mainly used to model shared resources among agents.
In the tragedy of commons example~\cite[Section 1.1.2]{nisan07}, agents share
a capped channel formulated as a shared constraint $\sum_{i=1}^N
x_i \le 1$.
Another example is the river basin pollution game
in~\cite{haurie97,krawczyk00}, where the total amount of pollutant thrown in
the river by the agents is restricted.
The environmental constraints are shared constraints in this case.
More details on how we model these examples can be found in
Section~\ref{subsec:exp-shared-constrs}.

There are two types of solutions when shared constraints are present.
Assume a suitable constraint qualification holds for each solution $x^*$
of Example~\ref{ex:shared-constr}.
Let $\mu^*_i$ be a multiplier associated with the shared constraint $h(x)$ for
agent $i$ at the solution $x^*$.
If $\mu_1^*=\cdots=\mu_N^*$, then we call the solution a
\textit{variational equilibrium}.
The name of the solution stems from the fact that if there are no $g_i(x)$'s,
then $x^*$ is a solution to the $\VI(X,F)$ and vice versa by
Proposition~\ref{prop:equilprob-vi}, where $X=\{x \in \real^n \mid h(x) \le
0\}$ and $h$ is a convex function.
In all other cases, we call a solution a GNEP equilibrium.

An interpretation from the economics point of view is that,
at a variational equilibrium, agents have the same marginal value on the
resources associated with the shared constraint (as the multiplier values are
the same), whereas at a GNEP equilibrium each agent may have a different
marginal value.

A shared constraint may not be easily modeled using the existing EMP
framework.
As each equation must be assigned to a single agent, we currently need to
create a replica of the shared constraint for each agent.
For Example~\ref{ex:shared-constr}, we may model it within GAMS/EMP as
follows:
\begin{lstlisting}[caption={Modeling the GNEP equilibrium via
replications},label={lst:gnep-shared-replication},language=GAMS]
variables obj(i), x(i);
equations deff(i), defg(i), defh(i);

model gnep_shared / deff, defg, defh /;

file empinfo / '%emp.info%' /;
put empinfo 'equilibrium' /;
loop(i,
  put 'min', obj(i), x(i), deff(i), defg(i), defh(i) /;
);
putclose empinfo;
\end{lstlisting}

In Listing~\ref{lst:gnep-shared-replication}, each {\tt defh(i)} is defined
exactly in the same way for all $i=1,\dots,N$: each of them is a replica of
the same equation.
This approach is neither natural nor intuitive compared to its algebraic
formulation.
It is also difficult to tell if the equation {\tt defh} is a shared constraint
by just reading the {\tt empinfo} file.
The information that {\tt defh} is a shared constraint is lost.
This could potentially prevent applying specialized solution methods, such
as the one in~\cite{schiro13}, for shared constraints.

Another difficulty lies in modeling the variational equilibrium.
To compute it, we need to have the multipliers associated with the shared
constraints the same among the agents.
Additional constraints may be required for such conditions to hold;
there is no easy way to force equality without changing the model using the
existing EMP framework.

\subsection{Extensions to model shared constraints}
\label{subsec:shared-constr-extension}

Our extensions have two new features:
i) we provide a syntactic enhancement that enables shared constraints to be
naturally and succinctly specified in a similar way to the algebraic formulation;
ii) we define a new EMP keyword that enables switching between the GNEP
and variational equilibrium without modifying each agent's problem definition.

To implement shared constraints, we modify Assumption~\ref{asm:original-emp}
as follows:

\begin{assumption}
  A model of an equilibrium problem described by equations and variables is
  assumed to have the following properties in the empinfo file:
  \begin{itemize}
  \item Each objective or \VI{} function of the model is owned by a single
    agent.
  \item Each constraint of the model is owned by at least one agent.
    If a constraint appears multiple times in different agents' problem
    definitions, then it is regarded as a shared constraint, and it is owned
    by these agents.
  \item Each variable is owned by a single agent.
  \end{itemize}
  \label{asm:shared-constr-emp}
\end{assumption}

Using Assumption~\ref{asm:shared-constr-emp}, we define shared constraints by
placing the same constraint in multiple agents' problems.
For example, we can model Example~\ref{ex:shared-constr} without replications
by changing lines 2 and 8-10 of Listing~\ref{lst:gnep-shared-replication} into
the following:
\begin{lstlisting}[caption={Modeling a shared constraint using a single
copy},label={lst:gnep-shared-noreplication},language=GAMS]
equation deff(i), defg(i), defh;

loop(i,
  put 'min', obj(i), x(i), deff(i), defg(i), defh /;
);
\end{lstlisting}

In Listing~\ref{lst:gnep-shared-noreplication}, a single instance of an
equation, \texttt{defh}, representing the shared constraint $h(x) \le 0$ is
created and placed in each agent's problem description.
Our framework then recognizes it as a shared constraint.
This is exactly the same way as its algebraic formulation is specified.
Also the \texttt{empinfo} file does not lose the problem structure: we can
easily identify that \texttt{defh} is a shared constraint by reading the
file as it appears multiple times.
To allow shared constraints, we need to specify \texttt{SharedEqu} in the
option file \texttt{jams.opt}.
Otherwise, multiple occurrences of the same constraint are regarded as an
error.
This is simply a safety check to stop non-expert users creating incorrect models.

In addition to the syntactic extension, we define a new EMP keyword
\texttt{visol} to compute a variational equilibrium associated with shared constraints.
By default, a GNEP equilibrium is computed if no \texttt{visol} keyword is
specified.
Hence Listing~\ref{lst:gnep-shared-noreplication} computes a GNEP
equilibrium.
If we place the following line in the \texttt{empinfo} file before the agents'
problem descriptions begin, that is, before line 3 in
Listing~\ref{lst:gnep-shared-noreplication}, then a variational equilibrium is
computed.
The keyword \texttt{visol} is followed by a list of shared constraints for
which each agent owning the constraint must use the same multiplier.

\begin{lstlisting}[caption={Computing a variational
equilibrium},label={lst:gnep-shared-visol},language=GAMS]
put 'visol defh' /;
\end{lstlisting}

Depending on the solution type requested, our framework creates different
\MCP{}s.
For a GNEP equilibrium, the framework replicates the shared constraint and
assigns a separate multiplier for each agent owning it.
For Example~\ref{ex:shared-constr}, the following $\MCP(z,F)$ is generated:
\begin{equation}
  \begin{aligned}
    &F(z) = ((F_i(z)^\tr)_{i=1}^N)^\tr, && z = ((z_i^\tr)_{i=1}^N)^\tr, \\
    &F_i(z) =
    \begin{bmatrix}
      \nabla_{x_i}f_i(x)-\nabla_{x_i}g_i(x)\lambda_i-\nabla_{x_i}h(x)\mu_i\\
      g_i(x)\\
      h(x)
    \end{bmatrix}, &&
    z_i =
    \begin{bmatrix}
      x_i\\
      \lambda_i \le 0\\
      \mu_i \le 0
    \end{bmatrix}, && \text{for} \quad i=1,\dots,N.
  \end{aligned}
  \label{eq:gnep-shared-mcp}
\end{equation}
Note that the same equation $h(\cdot)$ is replicated, and a separate
multiplier $\mu_i$ is assigned in~\eqref{eq:gnep-shared-mcp} for each agent
$i$ for $i=1,\dots,N$.

If a variational equilibrium is requested, then our framework creates a single
instance of the shared constraint, and a single multiplier is used for that
constraint among agents.
Accordingly, we construct the following $\MCP(z,F)$ for Example~\ref{ex:shared-constr}:
\begin{equation}
  \begin{aligned}
    &F(z)=((F_i(z)^\tr)_{i=1}^N, F_h(z)^\tr)^\tr, && z = ((z_i^\tr)_{i=1}^N,
    z_h^\tr)^\tr,\\
    &F_i(z)=
    \begin{bmatrix}
      \nabla_{x_i}f_i(x)-\nabla_{x_i}g_i(x)\lambda_i-\nabla_{x_i}h(x)\mu\\
      g_i(x)
    \end{bmatrix},
    && z_i=
    \begin{bmatrix}
      x_i\\\lambda_i \le 0
    \end{bmatrix}, && \text{for} \quad i=1,\dots,N,\\
    &F_h(z) =
    \begin{bmatrix}
      h(x)
    \end{bmatrix},
    && z_h =
    \begin{bmatrix}
      \mu \le 0
    \end{bmatrix}.
  \end{aligned}
  \label{eq:vi-shared-mcp}
\end{equation}

In~\eqref{eq:vi-shared-mcp}, a single multiplier $\mu$ is assigned to the
shared constraint $h(x)$, and $h(x)$ appears only once in the \MCP{}.
If there are no $g_i(x)$'s, then with a constraint qualification the problem
exactly corresponds to $\VI(X,F)$ of Proposition~\ref{prop:equilprob-vi} with
the set $X$ defined as $X:=\{x \mid h(x) \le 0\}$.

\subsection{Examples}
\label{subsec:exp-shared-constrs}

We present two GNEP examples having shared constraints in the following
sections, respectively.
The first example has a unique solution that is a variational equilibrium.
Thus, with or without the \texttt{visol} keyword, our framework computes the
same solution.
In the second example, multiple solutions exist.
Our framework computes solutions of different types depending on
the existence of the \texttt{visol} keyword in this case.

\subsubsection{GNEP with a shared constraint: tragedy of the commons}
\label{subsubsec:exp-gnep-shared-tragedy}

We consider the tragedy of the commons example~\cite[Section 1.1.2]{nisan07}:
\begin{equation}
  \begin{aligned}
    &\text{find} && (x_1^*,\dots,x_N^*) && \text{satisfying},\\
    &x_i^* \in && \argmax_{0 \le x_i \le 1} && x_i\left(1 -
      \left(x_i+\sum_{j=1, j\neq i}^Nx_j^*\right)\right),\\
    &&& \text{subject to} && x_i + \sum_{j=1,j\neq i}^N x_j^* \le 1.
  \end{aligned}
  \label{eq:exp-shared-ex1}
\end{equation}

There is a shared channel with capacity 1, represented as a shared constraint
$\sum_{j=1}^N x_j \le 1$, through which each agent $i$ sends $x_i$ units of
flow.
The value agent $i$ obtains by sending $x_i$ units is $x_i\left(1 -
  \sum_{j=1}^N x_j\right)$, and each agent tries to maximize its value.
By the form of the problem, \eqref{eq:exp-shared-ex1} is a GNEP with a shared
constraint.

The problem has a unique equilibrium $x_i^*=1/(N+1)$ for $i=1,\dots,N$.
The value of agent $i$ is then $1/(N+1)^2$, and the total value over all
agents is $N/(N+1)^2 \approx 1/N$.
As noted in~\cite{nisan07}, if agents choose to use $\sum_{i=1}^Nx_i = 1/2$,
then the total value will be 1/4 which is much larger than $1/N$ for large
enough $N$.
This is why the problem is called the tragedy of the commons.

We model~\eqref{eq:exp-shared-ex1} within GAMS/EMP in
Listing~\ref{lst:exp-shared-ex1}.
A single constraint \texttt{cap} is defined for the shared constraint,
and the same equation \texttt{cap} appears in each agent's problem
definition in the \texttt{empinfo} file.

\begin{lstlisting}[caption={Implementation of the
GNEP~\eqref{eq:exp-shared-ex1} within
GAMS/EMP},label={lst:exp-shared-ex1},language=GAMS]
$if not set N $set N 5

set i / 1*%N% /;
alias(i,j);

variables obj(i);
positive variables x(i);

equations defobj(i), cap;

defobj(i)..
    obj(i) =E= x(i)*(1 - sum(j, x(j)));

cap..
    sum(i, x(i)) =L= 1;

model m / defobj, cap /;

file info / '%emp.info%' /;
put info 'equilibrium' /;
loop(i,
    put 'max', obj(i), x(i), defobj(i), cap /;
);
putclose;

x.up(i) = 1;

* Specify SharedEqu option in the jams.opt file to allow shared constraints.
$echo SharedEqu > jams.opt
m.optfile = 1;

solve m using emp;
\end{lstlisting}

By default, a GNEP equilibrium is computed.
If we want to compute a variational equilibrium, we just need to place the
following line right after line 20 in Listing~\ref{lst:exp-shared-ex1}.
\begin{lstlisting}[language=GAMS]
put 'visol cap' /;
\end{lstlisting}

As the solution is unique $x_i^*=1/(N+1)$ with multiplier $\mu_i^*=0$ for
$i=1,\dots,N$, our framework computes the same solution in both cases.

\subsubsection{GNEP with shared constraints: river basin pollution game}
\label{subsubsec:exp-gnep-shared-river}

We present another example where we have different solutions for GNEP and
variational equilibria.
The example is the river basin example~\cite{haurie97,krawczyk00} described below:
\begin{equation}
  \begin{aligned}
    &\text{find} && (x_1^*,x_2^*,x_3^*) && \text{satisfying},\\
    &x_i^* \in &&\argmin_{x_i \ge 0} &&
    (c_{1i} + c_{2i}x_i)x_i - \left(d_1 - d_2\left(\sum_{j=1,j\neq i}^3 x^*_j
        + x_i\right)\right)x_i,\\
    &&&\text{subject to} &&
    \sum_{j=1,j\neq i}^3\left(u_{jm}e_jx_j^*\right) + u_{im}e_ix_i \le K_m,\\
    &&&&& \text{for} \quad m=1,2,\, i=1,2,3,\\
    &&&\text{where} && (c,d,e,u,K) \text{ is problem data}.
  \end{aligned}
  \label{eq:exp-shared-ex2}
\end{equation}
It has two shared constraints, and they are shared by all the three agents.

Let us briefly explain the model.
There are three agents near a river, each of which pursues maximum profit by
producing some commodities.
The term $(c_{1i}+c_{2i}x_i)x_i$ denotes the total cost of agent $i$, and
$(d_1 - d_2(\sum_{j=1,j\neq i}^3 x^*_j+ x_i))x_i$ is the revenue.
Each agent can throw pollutant in the river, but its amount is limited by
the two shared constraints in~\eqref{eq:exp-shared-ex2}.

Listing~\ref{lst:exp-shared-ex2} shows an implementation
of~\eqref{eq:exp-shared-ex2} within GAMS/EMP.
The two shared constraints are represented in the equations \texttt{cons(m)}.
We first compute a variational equilibrium.
A solution computed by our framework is $x^*=(21.145,16.028,2.726)$ with
multipliers $\mu_{\texttt{cons1}}^*=-0.574$ and $\mu_{\texttt{cons2}}^*=0$ for
the shared constraints \texttt{cons('1')} and \texttt{cons('2')},
respectively.\footnote{Note that we used the vector form for the constraints
  when we declare the equation \texttt{cons} for each agent in the
  \texttt{empinfo} file so that we do not have to loop through the set
  \texttt{m}.}

If we compute a GNEP equilibrium by deleting line 40 in
Listing~\ref{lst:exp-shared-ex2}, then we find a solution
$x^*=(0,6.473,22.281)$. In this case, multiplier values associated with the
shared constraints for each agent are as follows:
\begin{equation*}
  \begin{aligned}
    &\mu^*_{\texttt{cons1},1}=-0.804,\quad
    \mu^*_{\texttt{cons1},2}=-1.504,\quad \mu^*_{\texttt{cons1,3}}=-0.459,\\
    &\mu^*_{\texttt{cons2},1}=\mu^*_{\texttt{cons2},2}=\mu^*_{\texttt{cons2,3}}=0
  \end{aligned}
\end{equation*}

\begin{lstlisting}[caption={Implementation of~\eqref{eq:exp-shared-ex2} within
GAMS/EMP},label={lst:exp-shared-ex2},language=GAMS]
sets i / 1*3 /
     m / 1*2 /;
alias(i,j);

parameters
    K(m) / 1 100, 2 100 /
    d1   /   3     /
    d2   /   0.01  /
    e(i) / 1 0.5, 2 0.25, 3 0.75 /;

table c(m,i)
       1      2     3
  1  0.1   0.12  0.15
  2  0.01  0.05  0.01;

table u(i,m)
       1      2
  1  6.5    4.583
  2  5.0    6.250
  3  5.5    3.750;

variables obj(i);
positive variables x(i);

equations
    objdef(i)
    cons(m);

objdef(i)..
    obj(i) =E= (c('1',i) + c('2',i)*x(i))*x(i) - (d1 - d2*sum(j, x(j)))*x(i);

cons(m)..
    sum(i, u(i,m)*e(i)*x(i)) =L= K(m);

model m_shared / objdef, cons /;

file empinfo / '%emp.info%' /;
put empinfo 'equilibrium' /;
* Comment out the following line to compute a GNEP equilibrium.
put 'visol cons' /;
loop(i,
    put 'min', obj(i), x(i), objdef(i), 'cons' /;
);
putclose empinfo;

$echo SharedEqu > jams.opt
m_shared.optfile = 1;

solve m_shared using emp;

* Uncomment the code below to retrieve multipliers when a GNEP solution is computed.
* parameters cons_m(m,i);
* execute_load '%gams.scrdir%/ugdx.dat', cons_m=cons;
\end{lstlisting}

Note that since we only have a single constraint \texttt{cons} in the
modeling system, the lines 51-53 show how to recover a multiplier value for
each agent owning the shared constraint.

\section{Modeling equilibrium problems using shared
  variables}
\label{sec:impl}

In this section, we introduce {\it implicit variables} and their uses as
{\it shared variables}.
Roughly speaking, the values of implicit variables are implicitly defined by
other variable values.
Shared variables are implicit variables whose values are shared by multiple
agents.
For example, state variables controlled by multiple agents in economics, but
that need to have the same values across the problem, could be shared
variables.
In this case, our framework allows a single variable to represent such shared
variables.
This not only improves clarity of the model and facilitates deployment of
different mixed behavior models, but also provides a way of significantly
improving performance with efficient formulations.
In Section~\ref{subsec:shared-variable}, implicit variables and shared
variables are defined.
Section~\ref{subsec:mcp-formulations} presents various \MCP{} formulations for
them.
Finally, in Section~\ref{subsec:exp-impl}, we present examples of using shared
variables and experimental results comparing various \MCP{} formulations.

\subsection{Implicit variables and shared variables}
\label{subsec:shared-variable}

\begin{definition}
  We call a variable $y$ an implicit variable if for each $x$ there is at most
  one $y$ satisfying $(y,x) \in X$.
  Here the set $X$ is called the defining constraint of variable $y$.
  \label{def:implicit-variable}
\end{definition}

Note that Definition~\ref{def:implicit-variable} is not associated directly
with equilibrium problems.
It states that there exists one and only one implicit function $g(\cdot)$ such
that $(g(x),x) \in X$.
A simple example is $X=\{(y,x) \mid y=\sum_{i=1}^n x_i\}$.
We do not check for uniqueness however.
Our current implementation only allows the defining constraint $X$ to be
represented as a system of equations and the implicit variable $y$ to be
declared as a free variable.
They also need to be of the same size.
Constraints including bounds on variable $y$ can be introduced by explicitly
defining them in additional equations.
This is for allowing different solution types discussed in
Section~\ref{sec:shared-constr} to be associated with them.

Based on Definition~\ref{def:implicit-variable}, we define a shared variable.

\begin{definition}
  In equilibrium problems, variables $y_i$'s are shared variables if there is
  a set $X$ such that
  \begin{itemize}
  \item The feasible region of agent $i$ is given by
    \begin{equation}
      \begin{aligned}
        K_i(x_{-i}) &:= \{(y_i,x_i) \in \mathbb{R}^{n_y \times n_i}
        \mid (y_i,x_i) \in X_i(x_{-i}), (y_i,x_i,x_{-i}) \in X\},
        \text{ for }i=1,\dots,N.
      \end{aligned}
      \label{eq:shared-variable-replicate}
    \end{equation}
  \item $y_i$'s are implicit variables with the same defining constraint $X$.
  \end{itemize}
  \label{def:shared-variable}
\end{definition}

Basically, shared variables are implicit variables with
an additional condition that they have the same defining constraint.
One can easily verify that if $(y_1,\dots,y_N,x) \in K(x):=\prod_{i=1}^N
K_i(x_{-i})$, then $y_1=\cdots=y_N$, that is, variables $y_i$'s share their
values.
An extension to the case where they are shared by some subset of agents
is straightforward.

An equilibrium where shared variables $y_i$'s are present is defined as
follows:
\begin{equation}
  \begin{aligned}
    &\text{find} && (y^*,x_1^*,\dots,x_N^*,x_{N+1}^*)
    \quad \text{satisfying},\\
    &(y^*,x_i^*) && \in \argmin_{(y,x_i) \in K_i(x_{-i}^*)} \quad
    f_i(y,x_i,x_{-i}^*), \quad \text{for} \quad i=1,\dots,N,\\
    &x^*_{N+1} && \in \SOL(K_{N+1}(x_{-(N+1)}^*),F(\cdot, x_{-(N+1)}^*)).
  \end{aligned}
  \label{eq:equilprob-shared-var}
\end{equation}

Example~\ref{ex:shared-variable} presents the use of a shared variable
assuming that $y$ is an implicit variable with its defining constraint
$X:=\{(y,x) \mid H(y,x) = 0\}$.

\begin{example}
  The variable $y$ is a shared variable of the following equilibrium problem:
  \begin{equation*}
    \begin{aligned}
      &\text{find} && (y^*,x_1^*,\dots,x_N^*) && \text{satisfying},\\
      &(y^*,x_i^*) \in && \argmin_{y,x_i} && f_i(y,x_i,x^*_{-i}),\\
      &&&\text{subject to} && H(y,x_i,x_{-i}^*)=0, \text{ for }i=1,\dots,N,\\
      &&&\text{where} && H: \real^{m+n} \rightarrow \real^m, y \in \real^m
    \end{aligned}
  \end{equation*}
  \label{ex:shared-variable}
\end{example}

Listing~\ref{lst:impl-ex} presents GAMS code to model
Example~\ref{ex:shared-variable}.
We introduce a new keyword {\tt implicit} to declare an implicit variable and
its defining constraint.
The {\tt implicit} keyword is followed by a list of variables and
constraints, and our framework augments them to form a single vector of
implicit variables and its defining constraint.
It is required that the keyword should come first before any agent's problem
definition.
We can identify that $y$ is a shared variable in this case as it appears
multiple times in agents' problem definitions.
As the defining equation is assumed to belong to the implicit variable, we do
not place $H$ in each agent's problem definition (informally the variable $y$
owns $H$).

\begin{lstlisting}[caption={Modeling a shared
variable},label={lst:impl-ex},language=GAMS]
variables obj(i), x(i), y;
equations deff(i), defH;

model shared_implicit / deff, defH /;

file empinfo / '%emp.info%' /;
put empinfo 'equilibrium' /;
put 'implicit y defH' /;
loop(i,
   put 'min', obj(i), x(i), y, deff(i) /;
);
putclose empinfo;
\end{lstlisting}

Bounds on the shared variables can be introduced by explicitly defining them
in additional equations.
These bounds could change feasible region hence solutions of the problem.
For example, the following two-agent problem, each of which minimizes its
total cost, has bounds on the shared variable $y$, and its solution changes as
the value of the upper bound $b$ varies.
An implementation of~\eqref{eq:impl-ex-bounds} is available
at~\cite{youngdaeemp17}.
Our framework computes $(x_1^*,x_2^*)=(b/2,b/2)$ for $b \le 12$ and
$(x_1^*,x_2^*)=(6,6)$ for $b > 12$.
\begin{equation}
  \begin{aligned}
    &\text{find} && (y^*,x_1^*,x_2^*) && \text{satisfying},\\
    &(y^*,x_i^*) \in && \argmin_{x_i \ge 0, y} && x_i - x_i(10 -
    0.5*y),\\
    &&&\text{subject to}&& y = x_i + x_{-i}^*,\\
    &&&&& 0 \le y \le b, \,\,\text{for}\,\,i=1,2.
  \end{aligned}
  \label{eq:impl-ex-bounds}
\end{equation}

As we now allow shared variables, Assumption~\ref{asm:shared-constr-emp} needs
to be modified as follows:
\begin{assumption}
  A model of an equilibrium problem described by equations and variables is
  assumed to have the following properties in the empinfo file:
  \begin{itemize}
  \item Each \VI{} function of the model is owned by a single agent.
    Each objective function of the model is owned by at least one agent.
    The objective function can be owned by multiple agents when its objective
    variable is declared as an implicit variable.
  \item Each constraint of the model is owned by at least one agent.
    If a constraint appears multiple times in different agents' problem
    definitions, then it is regarded as a shared constraint owned by these
    agents.
  \item Each variable of the model is owned by at least one agent except for
    an implicit variable.
    If a variable appears multiple times in different agents' problem
    definition, then it is regarded as a shared variable owned by these
    agents, and it must be an implicit variable.
    If there is a variable not owned by any agent, then it must be an implicit
    variable.
  \end{itemize}
  \label{asm:shared-variable}
\end{assumption}

Note that in Assumption~\ref{asm:shared-variable} we allow missing ownership
for an implicit variable as its value is well-defined via its defining
constraint once the values of other variables are set.
When the ownership is not specified for an implicit variable, our framework
creates a \VI{} agent that owns the variable and its defining constraint: $H$
becomes a \VI{} function, and $y$ is its matching variable in
Example~\ref{ex:shared-variable}.
This turns out to be especially useful to model mixed behavior as described in
Section~\ref{subsubsec:mixed-behavior}.

\subsection{Various \MCP{} formulations for shared variables}
\label{subsec:mcp-formulations}

This section describes various \MCP{} formulations for shared variables.
For clarity, we will use Example~\ref{ex:shared-variable} to demonstrate our
formulations throughout this section.
Each formulation in
Sections~\ref{subsubsec:replicate}-\ref{subsubsec:substitute} shares the same
GAMS code of Listing~\ref{lst:impl-ex}.
Different formulations can be obtained by specifying an appropriate value
for the option \texttt{ImplVarModel} in the file {\tt jams.opt}.
In Section~\ref{subsec:exp-impl}, we present experimental results comparing
the sizes and performance of these formulations.

\subsubsection{Replicating shared variables for each agent}
\label{subsubsec:replicate}

In this reformulation, we replicate each shared variable for each agent owning
it and compute the corresponding \MCP{}.
For Example~\ref{ex:shared-variable}, our framework creates a variable $y_i$
for agent $i$, that is a replication of variable $y$, then computes the KKT
conditions.
The following $\MCP(z,F)$ is formulated by collecting these KKT conditions:
\begin{equation}
  \begin{aligned}
    &F(z) =
    \begin{bmatrix}
      (F_i(z)^\tr)_{i=1}^N
    \end{bmatrix}^\tr,
    && z=
    \begin{bmatrix}(z_i^\tr)_{i=1}^N
    \end{bmatrix}^\tr,\\
    &F_i(z) =
    \begin{bmatrix}
      \nabla_{x_i}f_i(x,y) - (\nabla_{x_i}H(y,x))\mu_i\\
      \nabla_{y_i}f_i(x,y) - (\nabla_{y_i}H(y,x))\mu_i\\
      H(y_i,x)
    \end{bmatrix},
    && z_i =
    \begin{bmatrix}
      x_i\\
      y_i\\
      \mu_i
    \end{bmatrix}.
  \end{aligned}
  \label{eq:impl-ex-repl-mcp}
\end{equation}

The size of~\eqref{eq:impl-ex-repl-mcp} is $(n+2mN)$ where the
first term is from $n=\sum_{i=1}^N|x_i|$ and the second one is from
$N\times(|y_i|+|\mu_i|)$ with $|y_i|=|\mu_i|=m$ for each $i=1,\dots,N$.
Note that the same constraints $H$ and shared variable $y$ are replicated $N$
times.
Table~\ref{tbl:impl-mcp-size} summarizes the sizes of the \MCP{} formulations
depending on the strategy.
\eqref{eq:impl-ex-repl-mcp} can be obtained by specifying an option
{\tt ImplVarModel=Replication} in {\tt jams.opt}.

\subsubsection{Switching shared variables with multipliers}
\label{subsubsec:switching}

We introduce a switching strategy that does not require replications.
The switching strategy uses the fact that in an \MCP{} we can exchange
free variables of the same size in the complementarity conditions without
changing solutions.
For example, if an \MCP{} is given by
\begin{equation*}
  \begin{aligned}
    \begin{bmatrix}
      F_1(z)\\
      F_2(z)
    \end{bmatrix}
    && \perp &&
    \begin{bmatrix}
      z_1\\
      z_2
    \end{bmatrix},
  \end{aligned}
\end{equation*}
where $z_i$'s are free variables, then a solution to the \MCP{} is a
solution to the following \MCP{} and vice versa:
\begin{equation*}
  \begin{aligned}
    \begin{bmatrix}
      F_1(z)\\
      F_2(z)
    \end{bmatrix}
    && \perp &&
    \begin{bmatrix}
      z_2\\
      z_1
    \end{bmatrix}.
  \end{aligned}
\end{equation*}

Applying the switching technique to shared variables, we switch each
shared variable with the multipliers associated with its defining equations.
This is possible because each shared variable is a free variable and its
defining equations are of the same size as the shared variable.
As a by-product, we do not have to replicate the shared variables and their
defining constraints.
Thus we can reduce the size of the resultant \MCP{}.

The $\MCP(z,F)$ obtained by applying the switching technique to
Example~\ref{ex:shared-variable} is as follows:
\begin{equation}
  \begin{aligned}
    &F(z) =
    \begin{bmatrix}
      (F_i(z)^\tr)_{i=1}^N, F_h(z)^\tr
    \end{bmatrix}^\tr,
    && z=
    \begin{bmatrix}(z_i^\tr)_{i=1}^N, z_h^\tr
    \end{bmatrix}^\tr,\\
    &F_i(z) =
    \begin{bmatrix}
      \nabla_{x_i}f_i(x,y) - (\nabla_{x_i}H(y,x))\mu_i\\
      \nabla_{y}f_i(x,y) - (\nabla_{y}H(y,x))\mu_i
    \end{bmatrix},
    && z_i =
    \begin{bmatrix}
      x_i\\
      \mu_i
    \end{bmatrix},\\
    &F_h(z) =
    \begin{bmatrix}
      H(y,x)
    \end{bmatrix},
    && z_h =
    \begin{bmatrix}
      y
    \end{bmatrix}.
  \end{aligned}
  \label{eq:impl-ex-switching-mcp}
\end{equation}

The size of~\eqref{eq:impl-ex-switching-mcp} is $(n+mN+m)$.
Note that compared to the replication strategy the size is reduced by
$(N-1)m$.
The number $(N-1)m$ exactly corresponds to the number of additional
replications of the shared variable $y$.
The formulation can be obtained by specifying an option
{\tt ImplVarModel=Switching} in {\tt jams.opt}.
This is currently the default value for \texttt{ImplVarModel}.

\begin{table}[t]
  \centering
  \begin{tabular}{|c|c|}
    \hline
    Strategy                & Size of the \MCP{}\\\hline
    replication             & $(n+2mN)$  \\
    switching               & $(n+mN+m)$ \\
    substitution (implicit) & $(n+nm+m)$ \\
    substitution (explicit) & $(n+m)$    \\
    \hline
  \end{tabular}
  \caption{The size of the \MCP{}s containing shared variables of
    Example~\ref{ex:shared-variable}}
  \label{tbl:impl-mcp-size}
\end{table}

\subsubsection{Substituting out multipliers}
\label{subsubsec:substitute}

We can apply our last strategy when the implicit function theorem holds for
the defining constraints.
By the implicit function theorem, we mean for $(\bar{y},\bar{x})$ satisfying
$H(\bar{y},\bar{x})=0$ there exists a unique continuously differentiable
function $h:\real^n \rightarrow \real^m$ that maps into some neighborhood of
$\bar{y}$ such that $H(h(x),x)=0$ for all $x$ in some neighborhood of
$\bar{x}$.

In a single optimization problem with $H$ taking the special form,
$H(y,x)=y-h(x)$, a similar definition was made in the AMPL modeling system,
and the variable $y$ is called a \textit{defined variable} in this
case~\cite[See A.8.1]{fourer02}.

The basic idea is to regard the shared variable $y$ as a function of other
non-shared variables and apply the total derivative.
At each solution $(y^*,x^*)$ of the problem, there exists a locally defined
implicit function $h_{x^*}(x)$ such that $y^*=h_{x^*}(x^*)$ and
$H(h_{x^*}(x),x)=0$ for each $x$ in some neighborhood of $x^*$ by the implicit
function theorem.
We can then remove variable $y$ by replacing it with the implicit function
$h_{x^*}(x)$ near $(y^*,x^*)$.
Thus the objective function $f_i(x_i,x_{-i},y)$ of agent $i$ on the feasible
set $H(y,x)=0$ near $(y^*,x^*)$ can be equivalently represented as
$f_i(x_i,x_{-i},h_{x^*}(x))$.
Consequently, the KKT conditions near $(y^*,x^*)$ only involve variable $x$:
\begin{equation*}
  \begin{aligned}
    \frac{d}{dx_i}f_i(x_i,x_{-i},h_{x^*}(x)) &=
    \nabla_{x_i}f_i(x_i,x_{-i},h_{x^*}(x)) +
    \nabla_{x_i}h_{x^*}(x)\nabla_{y}f_i(x_i,x_{-i},h_{x^*}(x)),\\
    y&=h_{x^*}(x),
  \end{aligned}
\end{equation*}
where $d/dx_i$ represents the total derivative with respect to variable $x_i$.

By the implicit function theorem, we have
\begin{equation*}
  \nabla_{x_i}h_{x^*}(x) = - \nabla_{x_i}H(y,x)\nabla_yH(y,x)^{-1}.
\end{equation*}
Therefore the KKT conditions of agent $i$'s problem of
Example~\ref{ex:shared-variable} can be represented as follows:
\begin{equation}
  \begin{aligned}
    &0 = \nabla_{x_i}f_i - \nabla_{x_i}H(\nabla_y H)^{-1}\nabla_yf_i
    &&\perp&& x_i && \text{free}, && \text{for} \quad i=1,\dots,N,\\
    &0 = H(y,x) &&\perp&& y && \text{free},
  \end{aligned}
  \label{eq:impl-ex-kkt-total-derivative}
\end{equation}
where we also applied the switching technique in
Section~\ref{subsubsec:switching}.

We can derive the same formulation~\eqref{eq:impl-ex-kkt-total-derivative}
from another perspective.
At a solution $(y^*,x^*,\mu^*)$ to the problem, the matrix
$\nabla_yH(y^*,x^*)$ is non-singular by the implicit function theorem.
Thus we have
\begin{equation}
  \begin{aligned}
    0 = \nabla_y f_i(x_i^*,x_{-i}^*,y^*) - (\nabla_yH(y^*,x^*))\mu^*_i
    &&\Longrightarrow&& \mu_i^* = (\nabla_y
    H(y^*,x^*))^{-1}\nabla_yf_i(x_i^*,x_{-i}^*,y^*).
  \end{aligned}
  \label{eq:impl-ex-kkt-mu-substitute}
\end{equation}
We can then substitute out every occurrence of $\mu_i$ by the right-hand side
of~\eqref{eq:impl-ex-kkt-mu-substitute} and remove the left-hand side from
consideration.
The result is the formulation~\eqref{eq:impl-ex-kkt-total-derivative}.

A critical issue with applying the
formulation~\eqref{eq:impl-ex-kkt-total-derivative} is that in general we do
not have the explicit algebraic representation of $(\nabla_yH)^{-1}$.
Computing it explicitly may be quite expensive and could cause numerical
issues.

Instead of explicitly computing it, we introduce new variables $\Lambda_i$ to
replace $\nabla_{x_i}H(\nabla_yH)^{-1}$ with a system of equations:
\begin{equation*}
  \begin{aligned}
    \Lambda_i \nabla_yH(y,x) = \nabla_{x_i}H(y,x), && \text{for} \quad
    i=1,\dots,N.
  \end{aligned}
\end{equation*}
One can easily verify that for each solution $(y^*,x^*)$
to~\eqref{eq:impl-ex-kkt-total-derivative} there exists $\Lambda_i^*$
satisfying the following and vice versa:
\begin{equation}
  \begin{aligned}
    &0 = \nabla_{x_i}f_i - \Lambda_i\nabla_yf_i
    &&\perp&& x_i && \text{free},\\
    &0 = \Lambda_i\nabla_yH - \nabla_{x_i}H
    &&\perp&& \Lambda_i && \text{free}, && \text{for} \quad i=1,\dots,N\\
    &0 = H(y,x) &&\perp&& y && \text{free}.
  \end{aligned}
  \label{eq:impl-ex-kkt-substitute-implicit}
\end{equation}
Consequently, the following $\MCP(z,F)$ is formulated in this case:
\begin{equation}
  \begin{aligned}
    &F(z) =
    \begin{bmatrix}
      (F_i(z)^\tr)_{i=1}^N, F_h(z)^\tr
    \end{bmatrix}^\tr,
    && z=
    \begin{bmatrix}(z_i^\tr)_{i=1}^N, z_h^\tr
    \end{bmatrix}^\tr,\\
    &F_i(z) =
    \begin{bmatrix}
      \nabla_{x_i}f_i(x,y) - \Lambda_i\nabla_yf_i(x,y)\\
      \Lambda_i\nabla_yH(y,x) - \nabla_{x_i}H(y,x)
    \end{bmatrix},
    && z_i =
    \begin{bmatrix}
      x_i\\
      \Lambda_i
    \end{bmatrix},\\
    &F_h(z) =
    \begin{bmatrix}
      H(y,x)
    \end{bmatrix},
    && z_h =
    \begin{bmatrix}
      y
    \end{bmatrix}.
  \end{aligned}
  \label{eq:impl-ex-substitute-implicit-mcp}
\end{equation}

The size of~\eqref{eq:impl-ex-substitute-implicit-mcp} is $(n+mn+m)$.
This could be much larger than the one obtained when we apply the switching
strategy, whose size is $(n+mN+m)$, because we usually have $n \gg N$.
Comparing the size to the case where we replicate the implicit variables, we
have $(n+nm+m) \le (n+2mN)$ if and only if $N \ge (n+1)/2$.

The size of the substitution strategy can be significantly reduced when the
shared variable is \textit{explicitly} defined, that is, $H(y,x)=y-h(x)$.
In this case, the algebraic representation of $(\nabla_yH)^{-1}$ is in a
favorable form: an identity matrix.
We do not have to introduce new variables and their corresponding system of
equations.
As we know the explicit algebraic formulation of $\nabla_{x_i}H$, the
following \MCP{} is formulated:
\begin{equation}
  \begin{aligned}
    &F(z) =
    \begin{bmatrix}
      (F_i(z)^\tr)_{i=1}^N, F_h(z)^\tr
    \end{bmatrix}^\tr,
    && z=
    \begin{bmatrix}(z_i^\tr)_{i=1}^N, z_h^\tr
    \end{bmatrix}^\tr,\\
    &F_i(z) =
    \begin{bmatrix}
      \nabla_{x_i}f_i(x,y) - \nabla_{x_i}H(y,x)\nabla_yf_i(x,y)
    \end{bmatrix},
    && z_i =
    \begin{bmatrix}
      x_i
    \end{bmatrix},\\
    &F_h(z) =
    \begin{bmatrix}
      H(y,x)
    \end{bmatrix},
    && z_h =
    \begin{bmatrix}
      y
    \end{bmatrix}.
  \end{aligned}
  \label{eq:impl-ex-substitute-explicit-mcp}
\end{equation}

Note that the size of~\eqref{eq:impl-ex-substitute-explicit-mcp} is $(n+m)$.
This is a huge saving compared to other formulations.
Our framework automatically detects if a shared variable is given in the
explicit form and substitutes out the multipliers if it is.
Otherwise, \eqref{eq:impl-ex-substitute-implicit-mcp} is formulated.
The formulation can be obtained by specifying an option
{\tt ImplVarModel=Substitution} in {\tt jams.opt}.

\subsection{Examples}
\label{subsec:exp-impl}

In this section, we introduce three models that use shared variables.
Section~\ref{subsubsec:sparsity} describes an example where we can reduce its
density significantly by introducing a shared variable.
This enables the problem, previously known as computationally intractable,
to be efficiently solved.
Section~\ref{subsubsec:epec} presents an EPEC model where each agent tries to
maximize its welfare in the Nash way while trading goods with other agents
subject to general equilibrium conditions.
The general equilibrium conditions define a set of state variables that are
shared by all agents.
We can then use the constructs for shared variables to define the state
variables.
In Section~\ref{subsubsec:mixed-behavior}, we present an example of modeling
mixed pricing behavior of agents.
More examples on using shared variables, for example modeling shared objective
functions, can be found at~\cite{youngdaeemp17}.
All experiments were performed on a Linux machine with Intel(R) Core(TM)
i5-3340M CPU@2.70 GHz processor and 8GB of memory.
\PATH{} was set to use the UMFPACK~\cite{umfpack} as its basis computation
engine.

\subsubsection{Improving sparsity using a shared variable}
\label{subsubsec:sparsity}

We consider an oligopolistic energy market equilibrium
example~\cite[Section 4]{luna14} formulated as a GNEP.
We show that its sparsity can be significantly improved by introducing a
shared variable, which makes the problem, known as computationally
intractable in~\cite{luna14}, solvable.
The example is defined as follows:
\begin{equation}
  \begin{aligned}
    & \text{find} && (q^*_0,q^*_1,\dots,q^*_5) && \text{satisfying},\\
    &q_0^* \in && \argmax_{0 \le q_0 \le U_0} &&
    p\left(\sum_{i=1}^5\sum_{k=1}^{n_i}q^*_{ik}\right)
    \left(\sum_{i=1}^5\sum_{k=1}^{n_i}q^*_{ik}\right)
    - \sum_{i=1}^5 c_i(q^*_i) - Pq_0,\\
    &&&\text{subject to} && q_0 + \sum_{i=1}^5\sum_{k=1}^{n_i}q^*_{ik} = d,\\
    &q_i^* \in && \argmax_{0 \le q_i \le U_i} &&
    p\left(\sum_{j=1,j\neq i}^5\sum_{k=1}^{n_j}q^*_{jk} +
      \sum_{k=1}^{n_i}q_{ik}\right)\left(\sum_{k=1}^{n_i}q_{ik}\right)
    - c_i(q_i),\\
    &&&\text{subject to} && q_0^* + \sum_{j=1,j\neq
      i}^5\sum_{k=1}^{n_j}q^*_{jk} + \sum_{k=1}^{n_i}q_{ik} = d,\\
    &&&\text{where} && c_i(q_i) = \frac{1}{2}q_i^\tr M_iq_i + b_i^\tr q_i,\\
    &&&&& p(Q) := \left(\frac{-P}{(1.5d)^2}Q^2 + P\right),\\
    &&&&& (P,d,M_i,b_i,U_i,n_i) \text{ is problem data},
    \quad \text{for} \quad i=1,\dots,5.
  \end{aligned}
  \label{eq:exp-luna}
\end{equation}

Let us briefly describe~\eqref{eq:exp-luna}.
There are six agents.
The first agent is an ISO agent which controls variable $q_0 \in \real$
measuring deficit of energy.
It tries to maximize the total profit of all the energy supplying agents less
the penalty caused by being unable to meet the fixed demand $d$.
The parameter $P$ represents how much penalty we put on the deficit $q_0$.
Each agent $i$, controlling $q_i=(q_{i1},\dots,q_{in_i})$ for $i=1,\dots,5$,
is a profit-maximizing agent that produces homogeneous energy generated from its $n_i$ number of plants.
Its decision variable $q_{ik}$ denotes the amount of energy produced from its
$k$th plant for $k=1,\dots,n_i$.
The function $p(Q)$ is a concave inverse demand function, and
$c_i(q_i)$ is the total cost of producing energy $\sum_{k=1}^{n_i}q_{ik}$.
The matrix $M_i$ is a diagonal matrix having positive diagonal entries, hence
$c_i(\cdot)$ is a strongly convex function.
All the six agents share the same demand constraint
$q_0+\sum_{i=1}^5\sum_{k=1}^{n_i}q_{ik}=d$; it is a shared constraint.
We use $n$, $n=\sum_{i=1}^5 n_i$, to denote the total number of plants, and
each energy-producing agent has the same number of plants, $n_i=n/5$ for
$i=1,\dots,5$.

In~\cite{luna14}, a variational equilibrium was computed by formulating a VI
and solving it using \PATH{}.
The paper reported that \PATH{} started to get much slower for the problem of
size $n=2,500$, and it was not able to solve problems of sizes
$n=5,000$ and $n=10,000$ due to out of memory error.

We have observed that the memory error was due to the high density of the
Jacobian matrix of the \MCP{}: it was almost 100\% for all problems.
Consequently, the \MCP{} will have a large number of nonzero entries
requiring a huge amount of memory.
Also the linear algebra computation (required by \PATH{} for basis
computations) time will be much slower in this case.

The root cause of such a highly dense Jacobian matrix was because of the term
$\sum_{i=1}^5\sum_{k=1}^{n_i} q_{ik}$ in the price function
$p(\cdot)$: for each $q_{ik}$, the term $\partial p(\cdot)/ \partial
  q_{ik}$ has all the variables $q_{i'k'}$.
We can make the problem much sparser by introducing a shared variable
$z:=\sum_{i=1}^5\sum_{k=1}^{n_i} q_{ik}$.
Mathematically, the problem is defined as follows:
\begin{equation}
  \begin{aligned}
    & \text{find} && (z^*,q^*_0,q^*_1,\dots,q^*_5) && \text{satisfying},\\
    &q_0^* \in && \argmax_{0 \le q_0 \le U_0} &&
    p(z^*)z^* - \sum_{i=1}^5 c_i(q^*_i) - Pq_0,\\
    &&&\text{subject to} && q_0 + z^* = d,\\
    &(z^*,q_i^*) \in && \argmax_{z,0 \le q_i \le U_i} &&
    p(z)\left(\sum_{k=1}^{n_i}q_{ik}\right) - c_i(q_i),\\
    &&&\text{subject to} && q_0^* + z = d,\\
    &&&&& z = \sum_{j=1,j\neq i}^5\sum_{k=1}^{n_j}q^*_{jk} +
    \sum_{k=1}^{n_i}q_{ik}.
  \end{aligned}
  \label{eq:exp-luna-shared}
\end{equation}

Listing~\ref{lst:exp-luna} implements~\eqref{eq:exp-luna-shared}.
We used the \texttt{visol} keyword to compute a variational equilibrium.
We formulate each agent's problem as a minimization problem by flipping the sign of its objective function.
Therefore, each agent $i$'s objective function for $i=1,\dots,5$ is strongly
convex, and the ISO agent's objective function is linear.

\begin{lstlisting}[caption={Implementation of~\eqref{eq:exp-luna-shared} using
a shared variable within GAMS/EMP},label={lst:exp-luna},language=GAMS]
$if not set n $set n 100
$if not set num_agents $set num_agents 5
$eval num_plants %n%/%num_agents%
$set P 120
sets i	/ 1*%num_agents% /
     k	/ 1*%num_plants% /;
alias(i,j);

variables iso_obj, agent_obj(i), z;
positive variables  q0, q(i,k);
equations iso_defobj, agent_defobj(i), demand, defz;
parameters U0, U(i,k), M(i,k), b(i,k), d, a;

U0 = 5;
U(i,k) = uniform(0,10);
M(i,k) = uniform(0.4,0.8);
b(i,k) = uniform(30,60);
d = 0.8 * sum((i,k), U(i,k));
a = -%P% / (1.5 * d)**2;

q0.up = U0;
q.up(i,k) = U(i,k);
q.l(i,k) = 0.8*U(i,k);
z.l = sum((i,k), q.l(i,k));

iso_defobj..
    iso_obj =E= %P%*q0
    + sum(i, 0.5*sum(k, M(i,k)*q(i,k)*q(i,k)) + sum(k, b(i,k)*q(i,k)))
    - (a*sqr(z) + %P%)*z;

agent_defobj(i)..
    agent_obj(i) =E=
    0.5*sum(k, M(i,k)*q(i,k)*q(i,k)) + sum(k, b(i,k)*q(i,k))
    - (a*sqr(z) + %P%)*sum(k, q(i,k));

demand..
    q0 + z =E= d;

defz..
    z =E= sum((i,k), q(i,k));

model m_oligop / iso_defobj, agent_defobj, demand, defz /;

file empinfo / '%emp.info%' /;
put empinfo 'equilibrium' /;
put 'implicit z defz' /;
put 'visol demand' /;
put 'min', iso_obj, q0, iso_defobj, demand /;
loop(i,
    put 'min', agent_obj(i);
    loop(k, put q(i,k););
    put z, agent_defobj(i), demand /;
);
putclose empinfo;

$echo SharedEqu > jams.opt
m_oligop.optfile = 1;

solve m_oligop using emp;
\end{lstlisting}

\begin{table}[t]
  \begin{subtable}{\textwidth}
    \centering
    \begin{tabular}{|r|r|r|r|r|r|r|}
      \hline
      \multicolumn{1}{|c|}{\multirow{2}{*}{$n$}}
      & \multicolumn{2}{c|}{Original}
      & \multicolumn{2}{c|}{Switching}
      & \multicolumn{2}{c|}{Substitution}\\\cline{2-7}
      & \multicolumn{1}{c|}{Size} & \multicolumn{1}{c|}{Density (\%)}
      & \multicolumn{1}{c|}{Size} & \multicolumn{1}{c|}{Density (\%)}
      & \multicolumn{1}{c|}{Size} & \multicolumn{1}{c|}{Density (\%)}\\\hline
      2,500 &   2,502 & 99.92 &  2,508 & 0.20 &  2,503 & 20.07\\
      5,000 &   5,002 & 99.96 &  5,008 & 0.10 &  5,003 & 20.04\\
      10,000 & 10,002 & 99.98 & 10,008 & 0.05 & 10,003 & 20.02\\
      25,000 &      - &     - & 25,008 & 0.02 &      - &     -\\
      50,000 &      - &     - & 50,008 & 0.01 &      - &     -\\
      \hline
    \end{tabular}
    \caption{\MCP{} model statistics when we have 1 ISO agent and 5
      energy-producing agents}
    \label{subtbl:exp-luna-orig-stat}
  \end{subtable}

  \vspace{2mm}
  \begin{subtable}{\textwidth}
    \centering
    \begin{tabular}{|r|r|r|r|r|r|r|}
      \hline
      \multicolumn{1}{|c|}{\multirow{2}{*}{$n$}}
      & \multicolumn{2}{c|}{Original}
      & \multicolumn{2}{c|}{Switching}
      & \multicolumn{2}{c|}{Substitution}\\\cline{2-7}
      & \multicolumn{1}{c|}{(Major,Minor)} & \multicolumn{1}{c|}{Time (secs)}
      & \multicolumn{1}{c|}{(Major,Minor)} & \multicolumn{1}{c|}{Time (secs)}
      & \multicolumn{1}{c|}{(Major,Minor)} & \multicolumn{1}{c|}{Time (secs)}
      \\\hline
       2,500  & (2,2639) &  57.78 &  (1,2630) &   1.30 &  (1,2630) &  13.18\\
       5,000  & (2,5368) & 420.92 &  (1,5353) &   5.83 &  (1,5353) &  91.01\\
      10,000  &        - &      - & (1,10517) &  22.01 & (1,10517) & 652.03\\
      25,000  &        - &      - & (1,26408) & 148.08 &         - &      -\\
      50,000  &        - &      - & (1,52946) & 651.14 &         - &      -\\
      \hline
    \end{tabular}
    \caption{Performance comparison when we have 1 ISO agent and 5
      energy-producing agents}
    \label{subtbl:exp-luna-orig-perf}
  \end{subtable}

  \vspace{2mm}
  \begin{subtable}{\textwidth}
    \centering
    \begin{tabular}{|r|r|r|r|r|}
      \hline
      \multicolumn{1}{|c|}{\multirow{2}{*}{$n$}}
      & \multicolumn{2}{c|}{Switching}
      & \multicolumn{2}{c|}{Substitution}\\\cline{2-5}
      & \multicolumn{1}{c|}{Size} & \multicolumn{1}{c|}{Density (\%)}
      & \multicolumn{1}{c|}{Size} & \multicolumn{1}{c|}{Density (\%)}
      \\\hline
       2,500  &  3,753 & 0.12 &  2,503 & 0.20 \\
       5,000  &  7,503 & 0.06 &  5,003 & 0.10 \\
      10,000  & 15,003 & 0.03 & 10,003 & 0.05 \\
      25,000  & 37,503 & 0.01 & 25,003 & 0.02 \\
      50,000  & 75,003 & 0.01 & 50,003 & 0.01 \\
      \hline
    \end{tabular}
    \caption{\MCP{} model statistics when we have 1 ISO agent and $n/2$
      energy-producing agents}
    \label{subtbl:exp-luna-sub-stat}
  \end{subtable}

  \vspace{2mm}
  \begin{subtable}{\textwidth}
    \centering
    \begin{tabular}{|r|r|r|r|r|}
      \hline
      \multicolumn{1}{|c|}{\multirow{2}{*}{$n$}}
      & \multicolumn{2}{c|}{Switching}
      & \multicolumn{2}{c|}{Substitution}\\\cline{2-5}
      & \multicolumn{1}{c|}{(Major,Minor)} & \multicolumn{1}{c|}{Time (secs)}
      & \multicolumn{1}{c|}{(Major,Minor)} & \multicolumn{1}{c|}{Time (secs)}
      \\\hline
       2,500  &  (1,2650) &   1.43 &  (1,2650) &   0.88 \\
       5,000  &  (1,5359) &   5.89 &  (1,5359) &   3.61 \\
      10,000  & (1,10526) &  25.05 & (1,10526) &  15.70 \\
      25,000  & (1,26400) & 176.94 & (1,26400) & 107.45 \\
      50,000  & (1,52950) & 800.75 & (1,52950) & 471.51 \\
      \hline
    \end{tabular}
    \caption{Performance comparison when we have 1 ISO agent and $n/2$
      energy-producing agents}
    \label{subtbl:exp-luna-sub-perf}
  \end{subtable}

  \caption{Model statistics and performance comparison of~\eqref{eq:exp-luna}
    using \PATH{}}
  \label{tbl:exp-luna}
\end{table}

Table~\ref{tbl:exp-luna} describes the statistics and performance
of~\eqref{eq:exp-luna} over various sizes of plants and agents.\footnote{The
  replication strategy is not allowed in this case as it is ambiguous what
  variable to replicate for the ISO agent: the agent uses shared variable $z$,
  but it does not own it. Our solver automatically detects this case and
  generates an error.}
The '-' symbol represents that we were not able to obtain the results
because of memory issue.
In Tables~\ref{subtbl:exp-luna-orig-stat} and~\ref{subtbl:exp-luna-orig-perf},
we used the same setup as in~\cite{luna14}.
First, note that the \MCP{} size of the original formulation was the
smallest, but it had the highest density.
This resulted in a computationally intractable model for large $n \ge
10,000$.
In contrast, using a shared variable and the switching strategy, we were able
to generate much sparser models and consequently to solve all of them.
However, the substitution strategy suffered a similar issue:
its high density generated computationally intractable models for $n=25,000$
and $50,000$.
This was due to the total derivative computation.
The term $\sum_{ik}q_{ik}$ remained in each component of the \MCP{} function
$F_i \in \mathbb{R}^{n_i}$ for each agent $i$.
This resulted in a block diagonal Jacobian matrix consisting of 5 100\% dense
blocks of size $n_i \times n_i$ for $i=1,\dots,5$.

To see the effect of many agents, we generated problems where each agent now
has 2 plants.
Thus for a given $n$ there are $n/2$ number of energy-producing agents.
Tables~\ref{subtbl:exp-luna-sub-stat} and~\ref{subtbl:exp-luna-sub-perf}
report the model statistics and performance comparison of the switching and
substitution strategies.
We did not report experimental results using the original formulation as the
\MCP{} size and the density of its Jacobian matrix were the same as before.
In this case, the substitution strategy showed the best performance.
Its Jacobian matrix was still block diagonal consisting of $n/2$ blocks, but
each block size was just $2 \times 2$.
This improved the sparsity of the model significantly.
The \MCP{} size of the switching strategy was much larger than that of the
substitution as its size is proportional to the number of agents (see
Table~\ref{tbl:impl-mcp-size}).
This made the strategy two times slower than the substitution strategy.

\subsubsection{Modeling equilibrium problems with equilibrium constraints}
\label{subsubsec:epec}

We construct an EPEC model\footnote{The original model was written by Thomas
  Rutherford, and was solved by applying the diagonalization method
  (Gauss-Seidel) to the nonlinear problem~\eqref{eq:epec} by fixing
  $t$ variable values belonging to other agents.
  We modified the model to use our EMP framework, and it was subsequently
  solved by \PATH{}.}
where data was taken from the GTAP (Global Trade Analysis Project) 9
database~\cite{aguiar16}.
The model is an exchange model having 23 agents (countries) where each agent
tries to maximize its welfare with respect to economic variables
(equivalently, state variables) and its strategic policy variables in the Nash
way while trading goods with other agents subject to the general equilibrium conditions.
Mathematically, the model is represented as follows:
\begin{equation}
  \begin{aligned}
    &\text{find} && (w^*,z^*,t^*) && \text{satisfying},\\
    &(w^*,z^*,t_i^*) \in && \argmax_{w,z,t_i \in T_i} && w_i,\\
    &&&\text{subject to} && H(w,z,t)=0,\\
    &&&&& \text{for} \quad i=1,\dots,23,
  \end{aligned}
  \label{eq:epec}
\end{equation}
where $w_i$ is a welfare index variable of agent $i$, $z$ is a vector of
endogenous economic variables such as prices, quantities, and so on, $t_i$
represents a vector of strategic policy variables of agent $i$ that determines
the tariffs on the imports, and $H(\cdot): \real^{253 \times 506} \rightarrow
\real^{253}$ is a system of nonlinear equations that represents the general
equilibrium conditions.

A distinguishing feature of the model is that the state variables $(w,z)$ are
shared by the agents, and their values are implicitly determined by the
general equilibrium conditions.
This implies that $(w,z)$ are shared variables, and the function $H$ is their
defining constraint.
In this case, $(w,z)$ are not given as an explicit function of $t$ in $H$.

In Table~\ref{tbl:epec}, we present experimental results of the three
formulations over various problem sizes by changing the number of agents.
The size of $H$ changes accordingly.
We use the replication strategy as a baseline to compare the size and
performance of the \MCP{} models.
We do not describe the implementation within GAMS/EMP as the number of lines
is long.
Refer to~\cite{youngdaeemp17} for the implementation.

\begin{table}[t]
  \centering
  \begin{tabular}{|r|r|r|r|r|r|r|}
    \hline
    \multicolumn{1}{|c|}{\multirow{2}{*}{\# Agents}}
    & \multicolumn{2}{c|}{Replication}
    & \multicolumn{2}{c|}{Switching}
    & \multicolumn{2}{c|}{Substitution} \\\cline{2-7}
    & Size & Density (\%) & Size & Density (\%) & Size & Density (\%) \\\hline
     5 &     570 & 1.66 & 350 & 3.34
       &   1,230 & 0.77 \\
    10 &   2,290 & 0.72 & 1,300 & 1.70
       &  10,210 & 0.14 \\
    15 &   5,160 & 0.50 & 2,850 & 1.28
       &  35,190 & 0.06 \\
    20 &   9,180 & 0.40 & 5,000 & 1.10
       &  84,420 & 0.03 \\
    23 &  12,144 & 0.37 & 6,578 & 1.03
       & 129,030 & 0.02 \\
    \hline
  \end{tabular}

  \vspace{5mm}
  \begin{tabular}{|r|r|r|r|r|r|r|}
    \hline
    \multicolumn{1}{|c|}{\multirow{2}{*}{\# Agents}}
    & \multicolumn{2}{c|}{Replication}
    & \multicolumn{2}{c|}{Switching}
    & \multicolumn{2}{c|}{Substitution} \\\cline{2-7}
    & (Major,Minor) & Time (secs) & (Major,Minor) & Time (secs)
    & (Major,Minor) & Time (secs)\\\hline
     5 & (18,164) &  0.33 & (18,173) & 0.22 &      (11,29) &   0.38 \\
    10 & (17,279) &  1.52 & (17,301) & 1.48 &     (15,436) &   8.14 \\
    15 &   (8,22) &  1.81 &   (8,22) & 1.68 &  (129,19806) & 814.73 \\
    20 &   (9,28) &  4.90 &   (9,28) & 4.73 &     (13,461) & 104.00 \\
    23 &   (9,41) & 10.07 &   (9,41) & 8.02 &    (20,1451) & 368.99 \\
    \hline
  \end{tabular}
  \caption{\MCP{} model statistics and performance comparison of the EPEC
    model}
  \label{tbl:epec}
\end{table}

In all settings, the switching strategy was of the smallest \MCP{} size
as it did not replicate or create any variables and equations.
Consequently, it showed the best performance in terms of the elapsed time: it
was up to 6 times faster than the replication strategy and 50
times\footnote{We did not include the 15-agent problem in the comparison as
  we think the slowest performance of the substitution strategy is due to some
  numerical difficulties \PATH{} encountered.} than the substitution strategy.
Although it performed more number of iterations on the problem having 10
agents, its time was still faster than that of the replication strategy.
We believe that the smaller problem size led to faster linear algebra
computation.

The substitution strategy was of the largest problem size and showed the slowest elapsed time.
The large size was due to the newly introduced variables and equations as
described in Table~\ref{tbl:impl-mcp-size}.
Although the density of it was the smallest, the number of nonzero entries was
the largest.
Hence linear algebra computation became much slower.

\subsubsection{Modeling mixed behavior: price-taking and
  price-making agents}
\label{subsubsec:mixed-behavior}

In this example, we show that mixed behavior of firms, switching between
price-takers and price-makers, can be easily modeled using a shared
variable.
We revisit the oligopolistic market equilibrium problem in
Section~\ref{subsubsec:exp-nash}.
Previously, the market was an oligopolistic market where all the firms were
price-makers: they can directly affect the price by changing their
productions.
If they have no control over the price, they become price-takers, that is, the
price is an exogenous variable for each firm.
In this case, the market is perfect competitive.

Listing~\ref{lst:exp-mixed} implements our mixed behavior model.
We introduce an implicit variable $z$ that represents the price $p(Q)$ defined
in~\eqref{eq:exp-nash}.
If a firm has ownership of variable $z$, then it becomes a price-maker as it
has a direct control of it.
Otherwise, it is a price-taker.
The first solve on line 32 computes a competitive market equilibrium.
As no firms have ownership of variable $z$, they are all price-takers in this
case.
After the first solve, we compute five different mixed models where firms
having indices less than or equal to $j$ are price-makers at the $j$th mixed
model for $j=1,\dots,5$.

\begin{lstlisting}[caption={Implementation of mixed behavior of agents within
GAMS/EMP},label={lst:exp-mixed},language=GAMS]
sets i agents / 1*5 /;
alias(i,j);

parameters c(i)    / 1  10, 2   8, 3   6, 4   4, 5   2 /
           K(i)    / 1   5, 2   5, 3   5, 4   5, 5   5 /
           beta(i) / 1 1.2, 2 1.1, 3 1.0, 4 0.9, 5 0.8 /;

variables obj(i), z;
positive variables q(i);

equations defobj(i), defz;

defobj(i)..
    obj(i) =e= q(i)*z - (c(i)*q(i) + beta(i)/(beta(i)+1)*K(i)**(-1/beta(i))*q(i)**((beta(i)+1)/beta(i)));

defz..
    z =e= 5000**(1.0/1.1)*sum(i, q(i))**(-1.0/1.1);

model mixed / defobj, defz /;

file empinfo / '%emp.info%' /;
put empinfo 'equilibrium' /;
put 'implicit', z, defz /;
loop(i,
    put 'max ', obj(i), q(i), defobj(i) /;
);
putclose empinfo;

q.l(i) = 10;
z.l = sum(i, q.l(i));

solve mixed using emp;

parameter objval(i,*), qval(i,*), pval(*), totalobjval(*), socialwelfare(*);

objval(i,'competitive') = obj.l(i);
qval(i,'competitive') = q.l(i);
pval('competitive') = z.l;
totalobjval('competitive') = sum(i, objval(i,'competitive'));
socialwelfare('competitive') = (5000**(1.0/1.1)*11*sum(i, q.l(i))**(0.1/1.1)
    - z.l*sum(i, q.l(i))) + totalobjval('competitive');

set kind / oligo1, oligo12, oligo123, oligo1234, oligo12345 /;

loop(kind,
    put empinfo 'equilibrium' /;
    put 'implicit', z, defz /;
    loop(i,
        if (i.val le ord(kind),
            put 'max ', obj(i), q(i), z, defobj(i) /;
        else
            put 'max ', obj(i), q(i), defobj(i) /;
        );
    );
    putclose empinfo;

    q.l(i) = 10;
    z.l = sum(i, q.l(i));

    solve mixed using emp;

    objval(i,kind) = obj.l(i);
    qval(i,kind) = q.l(i);
    pval(kind) = z.l;
    totalobjval(kind) = sum(i, objval(i,kind));
    socialwelfare(kind) = (5000**(1.0/1.1)*11*sum(i, q.l(i))**(0.1/1.1)
        - z.l*sum(i, q.l(i))) + totalobjval(kind);
);
\end{lstlisting}

Table~\ref{tbl:exp-mixed} presents firms' profits and social welfare of
various mixed models.
We computed social welfare by adding the consumer surplus to the total profit
of the firms.
The consumer surplus was computed by integrating the inverse demand function
less the amount paid by the consumer.
In columns starting with ``Oligo'', indices of firms that are price-makers
are attached to it.
Thus Oligo123 implies that firms with index 1, 2, or 3 are price-makers, and
others are price-takers.
As expected, i) the total profit of the firms was the smallest in the
competitive case and the largest in the oligopolistic case;
ii) each firm made more profit as it switched from price-taker to
price-maker;
iii) the social welfare was the maximized when all firms were price-takers.
Interestingly, switching from a price-taker to a price-maker of a firm
made profits of other firms increase much larger than the one of itself.
A similar observation was made in~\cite{harker88} and was explained as an
externality effect.

\begin{table}[t]
  \centering
  \begin{tabular}{|r|r|r|r|r|r|r|}
    \hline
    Profit & Competitive & Oligo1 & Oligo12 & Oligo123 & Oligo1234
    & Oligo12345 \\\hline
    Firm 1 & 123.834 & 125.513 & 145.591 & 167.015 & 185.958 & 199.934 \\
    Firm 2 & 195.314 & 216.446 & 219.632 & 243.593 & 264.469 & 279.716 \\
    Firm 3 & 257.807 & 278.984 & 306.174 & 309.986 & 331.189 & 346.590 \\
    Firm 4 & 302.863 & 322.512 & 347.477 & 373.457 & 376.697 & 391.279 \\
    Firm 5 & 327.591 & 344.819 & 366.543 & 388.972 & 408.308 & 410.357
    \\\hline
    Total profit & 1207.410 & 1288.273 & 1385.417 & 1483.023 & 1566.621
    & 1627.875\\
    Social welfare & 39063.824 & 39050.191 & 39034.577 & 39022.469 & 39016.373
    & 39015.125\\
    \hline
  \end{tabular}
  \caption{Profits of the firms and social welfare of various mixed models of
    Listing~\ref{lst:exp-mixed}}
  \label{tbl:exp-mixed}
\end{table}

\section{Modeling quasi-variational inequalities}
\label{sec:qvi}

This section introduces a new construct for specifying QVIs within our
framework and presents an example comparing two equivalent ways of defining
the equilibrium problems in either GNEP or QVI form.

\subsection{Specifying quasi-variational inequalities using our framework}
\label{subsec:qvi-specification}

Assuming that the feasible region of a $\QVI(K,F)$ takes the form $K(x)
:= \{ y \in \mathbb{R}^n \mid h(y,x) = 0, g(y,x) \le 0\}$,
Listing~\ref{lst:qvi} shows a generic way of specifying the $\QVI(K,F)$
using our framework.
In this case, we call $x$ a parameter variable and $y$ a variable of
interest.
Parameter variables could appear in the constraints, however, the \QVI{}
function $F$ must be defined by only variables of interest.

\begin{lstlisting}[caption={Modeling the QVI},label={lst:qvi},language=GAMS]
variables x(i), y(i);
equations defF(i), defh, defg;

* Definitions of defF(i), defh, and defg are omitted for expository purposes.

model qvi / defF, defh, defg /;

file empinfo / '%emp.info%' /;
putclose empinfo 'qvi defF y x defh defg';

solve qvi using emp;
\end{lstlisting}

To specify QVIs, the \texttt{empinfo} file starts with a new keyword
\texttt{qvi}.
The syntax is similar to the one for \VI{}s as described in
Section~\ref{subsec:underlying-assumptions} except that additional variables
could follow right after each function-variable pair.
In this case, these additional variables become parameter variables, and the
size of them must be the same as the size of variables of interest in the
preceding pair.
Our framework then constructs matching information between parameter
variables and variables of interest.
(The same applies to each preceding variable that is assigned to a zero
function.)
Therefore, in Listing~\ref{lst:qvi}, variables $y$ and $x$ are the variable of
interest and the parameter variable, respectively, and each $x_i$ is matched with $y_i$.
When our framework formulates the corresponding \MCP{}, for each constraint it
takes the derivative with respect to $y$, and each occurrence of $x_i$ is
replaced with $y_i$ using the matching information.
Note that if there are no parameter variables, that is, no variables follow
each function-variable pair and each preceding variable, then the problem
becomes a \VI{}.
In the above case, the feasible region is a fixed set, $K(x) := K$, specified
by \texttt{defh} and \texttt{defg}.

\subsection{Example}
\label{subsec:qvi-example}

We consider the following $\QVI(K,F)$ example in \cite[page 14]{outrata95}:
\begin{equation}
  \begin{aligned}
    &F(y) &&=
    \begin{bmatrix}
      -\frac{100}{3} + 2y_1 + \frac{8}{3}y_2\\
      -22.5 + \frac{5}{4}y_1 + 2y_2
    \end{bmatrix},\\
    &K(x) &&= \{ 0 \le y \le 11 \mid y_1 + x_2 \le 15, x_1 + y_2 \le 20\}
  \end{aligned}
  \label{eq:exp-qvi}
\end{equation}

Listing~\ref{lst:qvi} describes an implementation of~\eqref{eq:exp-qvi}.
As in~\eqref{eq:exp-qvi}, we use $x$ as a parameter variable in the
implementation.
The implementation is a natural translation of its algebraic form so that
users can focus on the QVI specification itself.
Also the {\tt empinfo} file retains information about variable types so that
we can easily identify which variables are parameter variables and which are
variables of interest.
This information can be potentially exploited for the efficient implementation
of solution methods for QVIs.
Our framework computes a solution $x^*=(10,5)$ that is consistent with the one
reported in~\cite{outrata95}.

\begin{lstlisting}[caption={Implementation of~\eqref{eq:exp-qvi} within
GAMS/EMP},label={lst:qvi},language=GAMS]
sets i / 1*2 /;
alias(i,j);

parameter A(i,j);
A('1','1') = 2;
A('1','2') = 8/3;
A('2','1') = 5/4;
A('2','2') = 2;

parameter b(i);
b('1') = 100/3;
b('2') = 22.5;

parameter Cy(i,j), Cx(i,j);
Cy(i,j)$(sameas(i,j)) = 1;
Cx(i,j)$(not sameas(i,j)) = 1;

parameter rhs(i) / 1 15, 2 20 /;

variables y(j), x(j);
equations F(i), g(i);

F(i)..
    sum(j, A(i,j)*y(j)) - b(i) =N= 0;

g(i)..
    sum(j, Cy(i,j)*y(j)) + sum(j, Cx(i,j)*x(j)) - rhs(i) =L= 0;

model qvi / F, g /;

file empinfo / '%emp.info%' /;
putclose empinfo 'qvi F y x g';

* If bounds on y and x are different, then an intersection of them is taken.
y.lo(j) = 0; y.up(j) = 11;
x.lo(j) = 0; x.up(j) = 11;

solve qvi using emp;
\end{lstlisting}

One can easily check that the \QVI{}~\eqref{eq:exp-qvi} is equivalent to the
GNEP~\eqref{eq:exp-gnep} in Section~\ref{subsubsec:exp-gnep} in terms of
solutions.
Actually, all the equilibrium examples described in previous sections can be
equivalently formulated as QVIs in the manner of Proposition~\ref{prop:equilprob-qvi}.

However, the information provided to our framework could be different
depending on the formulations.
The GNEP formulation~\eqref{eq:exp-gnep} gives us each agent's information:
its objective function and ownership of variables and constraints.
It may not be easy to recover this information from the QVI formulation.
In general, we can collect more information from an equilibrium formulation.
This could result in different solutions methods such as a Gauss-Seidel method
and its variants, while it may not be possible to collect similar information from the QVI formulation.
Therefore, for equilibrium problems, it may be better to not use the QVI
formulation.
Since our QVI framework is not just limited to QVIs derived from
equilibrium problems, it can be used to explicitly model other types of QVIs
with possible specialized algorithms for solution.

\section{Conclusions}
\label{sec:conclusions}

We have presented an extended mathematical programming framework for
equilibrium programming.
The framework defines a new set of constructs that enable equilibrium problems
with shared constraints and shared variables and their variational forms to be
specified in modeling languages.
Its syntax is a natural translation of the corresponding algebraic formulation
of the problem that captures high-level structure.
This allows more readable and less error prone models to be specified
compared to the traditional complementarity based models that require
the derivative computation of the Lagrangian by hand.
Different solution types such as variational equilibria associated with
shared constraints can be easily specified and computed using our framework.
We define shared variables and their associated constructs that can be used
to model sparse formulations, some forms of EPECs, price-taking and
price-making agents, shared objective functions, and so on.
We introduce a new construct for specifying QVIs.

There is potential for future work.
Using the high-level information captured by our framework, we can design
decomposition algorithms to solve large-scale equilibrium problems
that may involve a huge number of agents.
We intend to allow implicit variables to be defined using nonsmooth
equations~\cite{robinson13}.
We plan to extend our framework to incorporate equilibrium problems including
agents solving stochastic programs, bilevel programming, other forms of
EPECs, all with consideration of shared constraints and shared variables, and
to implement the EMP in other modeling systems such as AMPL and Julia.

\bibliographystyle{spmpsci}   
\bibliography{emp}       

\end{document}